\documentclass{article}

\usepackage{amsmath}
\usepackage{amssymb}
\usepackage{enumerate}
\usepackage{amsthm}
\usepackage{bbm}
\usepackage{listings}
\usepackage[all]{xy}
\usepackage{longtable}
\usepackage{graphicx}
\usepackage[backend=bibtex]{biblatex}
\newtheorem{theorem}{Theorem}[section]
\newtheorem{definition}[theorem]{Definition}
\newtheorem{remark}[theorem]{Remark}
\newtheorem{lemma}[theorem]{Lemma}
\newtheorem{corollary}[theorem]{Corollary}

\newcommand{\Mat}{\operatorname{Mat}}
\newcommand{\Hom}{\operatorname{Hom}}
\newcommand{\coker}{\operatorname{coker}}
\newcommand{\Span}{\operatorname{span}}
\newcommand{\rank}{\operatorname{rank}}
\newcommand{\codim}{\operatorname{codim}}

\newcommand{\D}{\operatorname{d}}
\newcommand{\CC}{\mathbb{C}}
\newcommand{\ZZ}{\mathbb{Z}}
\newcommand{\RR}{\mathbb{R}}
\newcommand{\PP}{\mathbb{P}}

\newcommand{\Sing}{\operatorname{Sing}}

\title{Vanishing cycles of smoothable isolated Cohen-Macaulay codimension $2$ singularities of type $2$}
\author{Matthias Zach\\
        Institut f. Alg. Geometrie, Leibniz Universt\"at Hannover, Germany}

\begin{document}

\maketitle

\begin{abstract}
  \noindent
  We extend the results from the previous paper by 
  A. Fr\"uhbis-Kr\"uger and the author \cite{FKZ15} to the vanishing topology of 
  those singularities in the title. 
  Studying the case of possibly non-isolated singularities in the Tjurina-transform,
  we reveal that in dimension $3$ and $2$ 
  there always is exactly one special vanishing cycle in degree $2$ 
  closely related to the determinantal structure of the singularity. 
\end{abstract}

\vspace{0.5cm}
The results of this paper give a detailed insight in the vanishing topology 
of a certain class of isolated determinantal singularities: Those given by the 
maximal minors of $3\times 2$ matrices with analytic entries. 
It turns out that the homology of the Milnor fiber reflects the determinantal 
structure, see theorem \ref{thm:MainTheorem}. We use the Tjurina modification and 
its compatibility with deformations as introduced in \cite{FKZ15} and combine 
it with results about the vanishing topology of non-isolated complete intersection
singularities. The latter are obtained as generalizations and adaptions of the 
work by D. Siersma, M. Tibar and Y. Yomdin 
\cite{Siersma83}, \cite{Siersma91}, \cite{ST15}, \cite{Iomdin74}. 
While computations of further examples beyond Cohen-Macaulay type $2$ yield 
that the observed phenomena still hold true for determinantal singularities 
defined by bigger matrices, the methods applied in this paper seem to be 
rather exhausted. 

The article is structured as follows. First we will review known facts for 
isolated Cohen-Macaulay codimension 2 singularities and the Tjurina modification 
in section 1. After stating the main theorem of this paper, we give an example 
to sketch the structure of the proof. In section 2 we develop the necessary 
results for the local complete intersection line singularities with a view 
towards their applications for determinantal ones. This will all be put 
together in section 3, where we compute the homology groups for the 
deformed Tjurina transform in a ``generic rank 1 perturbation'' and finally 
for the Milnor fiber of any Cohen-Macaulay codimension $2$ singularity of 
Cohen-Macaulay type $t=2$ and dimension $2$ or $3$.

The author wishes to thank A. Fr\"uhbis-Kr\"uger for guidance and support, M. Tibar 
for discussions during a visit in Hannover and the organization of the workshop 
on nonisolated singularities in Lille and D. Siersma for further discussions on the 
topic. Furthermore, M.A.S. Ruas and the ICMC 
for hospitality and a stimulating mathematical framework during a stay at 
the USP in S\~ao Carlos. Thanks also for conversations with T. Gaffney and 
M.A.S. Ruas on determinantal singularities, which significantly 
broadened the viewpoint of this paper.

\section{Known facts and techniques} 

\begin{definition}
  An isolated Cohen-Macaulay codimension $2$ (ICMC2) singularity 
  $(X,0) \subset (\CC^N,0)$ is, as the name yields, the germ of an analytic Cohen-Macaulay 
  scheme of codimension $2$ in $\CC^N$ such that any representative 
  $X$ of $(X,0)$ is smooth in a punctured neighborhood of the origin.
  The Cohen-Macaulay type of $(X,0)$ can be defined as 
  \[
    t := (\min \# \textnormal{generators of } I(X)) - 1,
  \]
  where $I(X)$ is the ideal in $\CC\{\underline x\}$ 
  associated to $(X,0)$ at $0$.
\end{definition}

The Hilbert Burch theorem says that being Cohen-Macaulay 
of codimension $2$ is equivalent 
to the fact that $I=I(X)$ is given by the $t$-minors of the 
syzygy-matrix $A$ of $I$, see \cite{Bur}. This matrix 
is always of the form 
\[
  A \in \Mat(t,t+1; \CC\{\underline x\}).
\]
Furthermore Schaps \cite{Sch} showed that any deformation of $(X,0)$, 
i.e. an embedding of $(X,0)$ in a flat family, is described as a 
perturbation of this matrix $A$. This means the minors 
of the perturbed matrix, whose entries now also depend on the 
deformation parameters, define the total space of the deformation.

In particular this means that every Cohen-Macaulay singularity is 
a \textit{determinantal singularity} of type $(t,t+1,t)$ in the following sense.

\begin{definition}
  A germ $(X,0) \subset (\CC^N,0)$ is called a determinantal singularity 
  of type $(m,n,s)$, 
  if there exists a holomorphic map germ 
  \[
    A : (\CC^N,0) \to (\Mat(m,n; \CC), 0)
  \]
  such that $(X,0)$ appears as the preimage of the 
  \textit{generic determinantal variety} 
  \[
    M_{m,n}^s := \{ B \in \Mat(m,n; \CC) : \rank B < s \} \subset \Mat(m,n;\CC)
  \]
  under this map
  \[ 
    (X,0) = A^{-1}( M_{m,n}^s, 0 ),
  \]
  and $(X,0)$ has 
  \textit{expected codimension} 
  $\codim (X,0) = \codim M_{m,n}^s$.
\end{definition}

\noindent
By definition a deformation of a determinantal singularity comes from a perturbation of 
its matrix $A$ as a map germ. The condition on $(X,0)$ to have expected 
codimension assures that the induced family for the singularity is flat.
Thus also the notions of deformation for (isolated) CMC2 and matrix 
singularities agree. 

If a determinantal singularity $(X,0)$ is isolated, it also automatically is an 
\textit{EIDS} in the sense of Ebeling and Gusein-Zade \cite{EGZ09}.
Recall that the varieties 
\[
  \{0\} = M_{m,n}^0 \subset M_{m,n}^1 \subset \cdots \subset M_{m,n}^{\min\{m,n\}} 
  \subset \Mat(m,n; \CC )
\]
give a canonical Whitney stratification of the space of $(m\times n)$-matrices.

\begin{definition}
  A determinantal singularity $(X,0)\subset (\CC^N,0)$ given by a matrix 
  $A: (\CC^N,0) \to M_{m,n}$ as $A^{-1}(M_{m,n}^s)$ is an \textit{essentially isolated 
  determinantal singularity} (EIDS), if the map $A$ is transverse 
  to all strata of $M_{m,n}^s$ in a punctured neighborhood of the origin.
\end{definition}

In general, determinantal singularities do not admit smoothings, but only stabilizations. 
These are deformations coming from a perturbations $A_\varepsilon$ of the defining matrix
such that considered as a map, $A_\varepsilon$ 
is transversal to all strata 
$M_{m,n}^s$ of $\Mat(m,n;\CC)$. 
From this it is easy to see that a determinantal singularity
$(X,0) \subset (\CC^N,0)$ of type $(m,n,s)$ admits a smoothing, if and only if 
\[
  N < \codim M_{m,n}^{s-1},
\]
so that we have enough degrees of freedom to move the image of $A$ away from the 
lower dimensional strata.
\begin{definition}
  Let $0 \in B \subset \CC^N$ be a Milnor ball for a representative $X$ of 
  a smoothable isolated determinantal singularity $(X,0)\subset (\CC^N,0)$ given by a matrix 
  $A : B \to M_{m,n}$ as $X = A^{-1}(M_{m,n}^s)$. Let 
  \[
    A_\varepsilon : B \to M_{m,n}
  \]
  be a stabilization of $A$. The space $X_\varepsilon = A_\varepsilon^{-1}(M_{m,n}^s)$ 
  is the \textit{Milnor fiber} of $(X,0)$.
  The generators of the homology groups $H_i(X_\varepsilon)$ are the 
  \textit{vanishing cycles} and the reduced Euler characteristic 
  $\overline \chi(X_\varepsilon)$ the \textit{vanishing Euler characteristic} of 
  $(X,0)$.
  \label{def:SittingOver}
\end{definition}

\noindent
Throughout this paper we will only use homology and cohomology with 
integer coefficients. Hence we will omit them from the notation and just 
write $H_q(X)$ for $H_q(X;\ZZ)$ and vice versa in cohomology.

One can use the theory of versal unfoldings for the map germ $A$ to show 
that the diffeomorphism type of the Milnor fiber is unique. Thus the 
homology groups of $X_\varepsilon$ and its invariants are in fact invariants 
of the singularity $(X,0)$ itself. 

\begin{remark}
  Let $(X,0) \subset (\CC^{n+d},0)$ be an isolated complete intersection 
  singularity (ICIS) of codimension $d$.
  $(X,0)$ can be seen as a determinantal singularity of type $(d,1,1)$. 
  It is known \cite{Hamm72} that its Milnor fiber $X_\varepsilon$ is homotopic to a 
  bouquet of spheres 
  \[
    X_\varepsilon \cong S^n \vee \dots \vee S^n
  \]
  of real dimension $n$. Hence besides $b_0$ 
  there is only the middle Betti number $b_n(X_\varepsilon)$, which is nonzero. 
  It is also known as the \textit{Milnor number} $\mu(X,0)$ of the 
  singularity.
\end{remark}

Greuel and Steenbrink showed in \cite{GS} that the Betti numbers $b_i$ 
of the Milnor fiber $X_\varepsilon$ 
of any smoothable isolated singularity $(X,0) \subset (\CC^N,0)$ must 
be zero in the range $0<i\leq\dim X - \codim X$. 
For a smoothable ICMC2 singularity of dimension $n=\dim (X,0) \geq 2$ 
this means that
there are two possibly nonzero 
Betti numbers $b_n$ and $b_{n-1}$ of the Milnor fiber $X_\varepsilon$. 
However it is also shown in \cite{GS}, that for surfaces
$b_1$ is always zero.
It is an open question, whether the first homology group is really 
zero or if there is torsion, but in general the vanishing cycles for 
ICMC2 surface singularities were not expected to behave very different from 
the complete intersection case. 

For threefolds there are examples, 
for which this is no longer the case, as was first observed by 
James Damon and Brian Pike in \cite{DP2}. 
Using Macaulay2, they computed the reduced Euler characteristic 
\[
  -\overline \chi(X_\varepsilon) = b_3 - b_2
\]
of the Milnor fiber 
and showed that it is negative for certain examples 
from the list of simple ICMC2 singularities in \cite{FN}. 

Methods for the computation of the Euler characteristic have also been 
developed in \cite{EGZ09}, \cite{GafRua16} and \cite{NOT}. 
Those using polar multiplicities can also be effectively implemented 
in Singular to determine the Euler characteristic of given examples.
However neither of the methods is suited to compute the two Betti numbers 
or even the distinct homology groups independently.

This was first done by Anne Fr\"uhbis Kr\"uger and the author in \cite{FKZ15} 
for the simple ICMC2 threefolds. 
The key tool in that paper was the \textit{Tjurina modification}, which we will 
briefly review in the next section. Although the original intention in 
\cite{FKZ15} was to explain the vanishing topology of threefolds, the ideas 
and techniques are also applicable in any other dimension greater than zero 
and beyond the ICMC2 case for determinantal singularities in general, see e.g. 
\cite{Helge16}.

\subsection{Tjurina transform for ICMC2 singularities}

Let $(X_0,0)\subset (\CC^N,0)$ be a determinantal singularity 
of type $(m,n,t)$ given by the matrix 
$A \in \Mat(m,n;\CC\{\underline x\})$. Consider $A$ as a map germ 
\[
  A : (\CC^N, 0) \to (M_{m,n},0).
\]
The generic determinantal variety $M_{m,n}^t$ is singular along the 
set $M_{m,n}^{t-1}$. Let 
\[
  L : M_{m,n}^{t} \dashrightarrow \operatorname{Gr}(t-1,n),  
  \quad B \mapsto \Span(B^T)
\]
be the rational map to the Grassmannian of $(t-1)$-dimensional 
subspaces of $\CC^m$, mapping a matrix $B$ to the plane spanned by the 
image of its transpose. It is well defined on $M_{m,n}^t$ away from 
its singular locus $M_{m,n}^{t-1}$. 
The blowup of $L$ 
\[
  \hat M_{m,n}^{t} := \overline{\Gamma_L(M_{m,n}^t\setminus M_{m,n}^{t-1})} 
  \subset \Mat(m,n; \CC) \times \operatorname{Gr}(t-1,n)
\]
is as usual defined to be the closure of the graph 
$\Gamma_L$ of $L$ over the regular part. 
It is a resolution 
\[
  \rho : \hat M_{m,n}^{t} \to M_{m,n}^t
\]
of $M_{m,n}^t$ with exceptional locus $ \operatorname{Gr}(t-1,n)$ over 
every point of $M_{m,n}^{t-1}$.

\begin{definition}
  For a determinantal singularity $(X_0,0) \subset (\CC^N,0)$ 
  of type $(m,n,t)$ given by a matrix $A$, the Tjurina transform 
  \[
    (Y_0,V) \subset 
    (\CC^N \times \operatorname{Gr}(t-1,n), \{0\} \times \operatorname{Gr}(t-1,n))
  \]
  is defined as the fiber product 
  \[
    \xymatrix{
      X_0 \times_{M_{m,n}^t} \hat M_{m,n}^t \ar[r] \ar[d]_\pi & 
      \hat M_{m,n}^t \ar[dr]^{\hat L} \ar[d]_\rho & \\
      X_0 \ar[r]^A & 
      M_{m,n}^t \ar@{-->}[r]^L & 
      \operatorname{Gr}(t-1,n)
    }
  \]
\end{definition}

For ICMC2 singularities the target Grassmannian is always 
\[
  \operatorname{Gr}(t-1,t) \cong \PP^{t-1},
\]
where we identify a $(t-1)$-plane with the class of its normal vector in 
the dual space.
The equations for the Tjurina transform take a surprisingly simple form, see 
\cite{FKZ15}, corollary 3.3. Let 
$\underline x = (x_1,\dots,x_N)$ be the local coordinates of $\CC^N$ and 
$(s_1:\dots:s_t)$ the homogeneous coordinates of $\PP^{(t-1)}$.
If $A = (a_{i,j})_{i,j}$ was the matrix describing $(X_0,0)$, 
then $Y_0$ is the zero locus of the equations
\begin{equation}
  \begin{pmatrix}
    f_0 \\ \vdots \\ f_t
  \end{pmatrix} = 
  \begin{pmatrix}
    a_{1,1}(\underline x)& \cdots & a_{1,t}(\underline x)\\
    \vdots & & \vdots \\
    a_{t+1,1}(\underline x)& \cdots & a_{t+1,t}(\underline x)\\
  \end{pmatrix}
  \cdot 
  \begin{pmatrix}
    s_0 \\ \vdots \\ s_t
  \end{pmatrix}
  = 
  \begin{pmatrix}
    0 \\ \vdots \\ 0
  \end{pmatrix}
  \label{eqn:EquationsTjurinaTransform}
\end{equation}
in $\CC^N\times \PP^{(t-1)}$.
The excetional set $V$ of the projection $\pi : Y_0 \to X_0$ is by this construction 
always the 
whole $\PP^{(t-1)}$. Hence $Y_0$ might not be equidimensional if the Cohen-Macaulay 
type of $(X_0,0)$ was too big. 
However counting dimensions one easily verifies that $Y_0$ 
is a local complete intersection iff $\dim (X_0,0) \geq t-1$.

Because any deformation of an ICMC2 singularity $(X_0,0)$ comes from a perturbation 
of the defining matrix, Tjurina modification is also applicable in family. 
This basically means, we take the Tjurina transform $Y$ of a representative 
$X$ of the total space
of a deformation of $(X_0,0)$, see \cite{FKZ15}, construction 3.6. 
Thus every deformation of $(X_0,0)$ over a base $B$ canonically induces a family 
\[
  \xymatrix{
    Y_0 \ar@{^(->}[r] \ar[d]_{\pi_0} & 
    Y \ar[d]^\pi \\
    X_0 \ar@{^(->}[r] \ar[d] & 
    X \ar[d] \ar[d]^\varepsilon \\
    \{0\} \ar@{^(->}[r] &
    B
  }
\]
with the Tjurina transform $Y_0$ as a special fiber. 
If $\dim (X_0,0) \geq t-1$, i.e. if $Y_0$ is a local complete intersection, this 
family is automatically flat (\cite{FKZ15}, proposition 3.9).

In this setup, we can study deformations of 
$(X_0,0)$ via the deformations of $(Y_0,V)$ sitting over it. 
Moreover in a smoothing of $(X_0,0)$, the induced map 
\[
  \pi_\varepsilon : Y_\varepsilon \to X_\varepsilon
\]
on the fibers over $\varepsilon \neq 0$ is always an isomorphism 
(\cite{FKZ15}, proposition 3.8) -- a direct consequence of the fact 
that the matrix $A(x)$ cannot degenerate on smooth points $x \in X_\varepsilon$.

\subsection{Previously known results and main theorem}

The main theorem about the topology from  
\cite{FKZ15} can be summarized as follows:

\begin{theorem}{(\cite{FKZ15}, theorem 4.4)}
  Let $(X_0,0) \subset (\CC^5,0)$ be an ICMC2 threefold singularity 
  of Cohen-Macaulay type $2$ such that the 
  Tjurina transform $(Y_0,V)$ has at most isolated singularities. Then the 
  Betti numbers of the Milnor fiber $X_\varepsilon$ are given by 
  \[
    b_0 = 1, \quad
    b_1 = 0, \quad
    b_2 = 1, \quad 
    b_3 = r,
  \]
  where $r$ is the sum of the Milnor numbers of the ICIS in $Y_0$.
  \label{thm:MainTheoremLastPaper}
\end{theorem}

\noindent
We will generalize this result in two directions. First of all 
the ideas and the proof for theorem \ref{thm:MainTheoremLastPaper} 
in \cite{FKZ15} carry over 
almost literally to the surface case.

\begin{theorem}
  Let $(X_0,0) \subset (\CC^4,0)$ be an ICMC2 surface singularity 
  of Cohen-Macaulay type $2$ such that the 
  Tjurina transform $(Y_0,V)$ has at most isolated singularities. 
  Then the Milnor fiber $X_\varepsilon$ is simply connected. 
  The second homology group splits into 
  \begin{equation}
    H_2(X_\varepsilon) \cong \ZZ^r \oplus \ZZ
    \label{eqn:SplittingSecondHomologyGroupSurfaces}
  \end{equation}
  where all cycles of the first summand come from the ICIS in 
  the Tjurina transform and $r$ is the sum of their Milnor numbers.
  \label{thm:MainTheoremLastPaperSurfaces}
\end{theorem}
We give a sketch of the proof in order to also recall the key arguments.
\begin{proof}
  The exceptional set $V = \{0\}\times \PP^1$ of the Tjurina transform $Y_0$ 
  is a strong deformation retract of $Y_0$. Let $B = \bigcup_{i=1}^N$ be a 
  union of Milnor balls around the singular points $p_i$ of $Y_0$. 
  For $q>0$ we have isomorphisms 
  \[
    H_q(V) \cong H_q(Y_0) \cong H_q(Y_0,B)
  \]
  and if $q=2$, $H_2(Y_0,B)$ is freely generated by a relative cycle, 
  which can be represented by the 
  exceptional set $V \cong \PP^1$ with the interiors of the Milnor balls 
  $B_i$ cut out from it. It is a sphere with holes and the boundary consists 
  of circles located in the links of the ICIS of $Y_0$. 

  When we now pass to a smoothing $Y_\delta$ of $Y_0$, we do not 
  change the topology outside the Milnor balls $B_i$. Glueing in the 
  local Milnor fibers $F_i \subset B_i$, we obtain the following 
  long exact sequence
  \[
    \xymatrix{
      0 \ar[r] & 
      H_2(\bigcup_{i} F_i ) \ar[r] & 
      H_2(Y_\delta) \ar[r] & 
      H_2(Y_\delta, \bigcup_i F_i ) \ar[r] &
      H_1(\bigcup_i F_i ) \ar[r] & 
      \cdots
    }
  \]
  By excision we have $H_2(Y_\delta,\bigcup_i F_i) \cong H_2(Y_0,B) \cong \ZZ$. 
  The zero on the left shows that there are no relations among the vanishing 
  cycles in the local Milnor fibers $F_i$. On the right, the term 
  $H_1(\bigcup_i F_i)$ vanishes because of Hamms result \cite{Hamm72}. 
  We deduce the desired splitting.
\end{proof}

For ICMC2 singularities $(X_0,0)$ for which the Tjurina transform 
is smooth, the Milnor fiber is diffeomorphic to $(Y_0,V)$. 
Consequently if we let 
\[
  L : X_\varepsilon \to \PP^1, \quad x \mapsto \Span A_\varepsilon^T(x),
\]
be the regular map on the Milnor fiber given by the deformed matrix 
$A_\varepsilon$, then a generator of 
$H_2(X_\varepsilon)$ is given by the fundamental class of a differentiable section 
$l : \PP^1 \to X_\varepsilon$ 
of $L$, i.e. a map $l$ such that $L \circ l = \operatorname{Id}_{\PP^1}$.

In general the existence of such a section is hard to prove. But from the proof 
of theorem \ref{thm:MainTheoremLastPaperSurfaces} it is evident that the generator 
of the second summand of the splitting (\ref{eqn:SplittingSecondHomologyGroupSurfaces}),
or just the generator of the second homology group in theorem 
\ref{thm:MainTheoremLastPaper}, 
is ``coming from'' the exceptional set. To make this more precise, we give the 
following definition.

Let $(X,0) \subset (\CC^N,0)$ be a determinantal singularity of type $(m,n,t)$ given 
by a matrix $A$ and 
\[
  A_\varepsilon : B \to M_{m,n}
\]
a stabilization of $A$ defined on some Milnor ball $B \subset \CC^N$ for 
$(X,0)$. Because $A_\varepsilon$ is transverse to all the strata of 
$M_{m,n}^s$, the Tjurina transform 
$Y_\varepsilon \subset B \times \operatorname{Gr}(t-1,n)$
of $X_\varepsilon = A_\varepsilon^{-1}(M_{m,n}^s)$ 
is a smooth compact manifold with corners
(Recall that $Y_\varepsilon$ is isomorphic to $X_\varepsilon$ in case $(X,0)$ 
was smoothable). By abuse of notation, let 
\[
  L: Y_\varepsilon \subset B \times \operatorname{Gr}(t-1,n) \to 
  \operatorname{Gr}(t-1,n)
\]
be the projection to the Grassmannian. Consider the image 
$G \subset H^\bullet(Y_\varepsilon)$ of 
the induced map 
\[
  L^* : H^\bullet(\operatorname{Gr}(t-1,n)) \to H^\bullet(Y_\varepsilon)
\]
in cohomology. 

\begin{definition}
  A cycle $[\sigma] \in H_\bullet(Y_\varepsilon)$ is said to be 
  \textit{horizontal}, if the cap product 
  $ g \cap [\sigma] $ is zero for all $g \in G = L^*(H^\bullet(\operatorname{Gr}(t-1,n)))$.
  We also write 
  \[
    [\sigma] \in G^\perp.
  \]
  All other cycles in $H_\bullet(Y_\varepsilon)$ are called \textit{vertical}. 
  We also say they are \textit{sitting over} the Grassmannian.
\end{definition}

\begin{corollary}
  Let $X_\varepsilon$ be the Milnor fiber of an ICMC2 singularity 
  $(X_0,0)\subset (\CC^{n+2},0)$ of dimension $n=2$ or $3$ and of 
  Cohen-Macaulay type $t=2$ with only isolated singularities in 
  the Tjurina transform. Then the homology of $X_\varepsilon$ splits 
  into 
  \[
    H_\bullet(X_\varepsilon) \cong G^\perp \oplus \ZZ,
  \]
  where the second summand lives in degree $2$ and the cap product with 
  $L^*([\PP^1]^\vee)$, the pullback of the dual of the fundamental class of 
  $\PP^1$, is a perfect pairing.
  \label{cor:SplittingHomologyGroupsOldPaper}
\end{corollary}

The main goal of this paper is to extend 
theorem \ref{thm:MainTheoremLastPaper},  
theorem \ref{thm:MainTheoremLastPaperSurfaces} and 
corollary \ref{cor:SplittingHomologyGroupsOldPaper}
to the case of arbitrary ICMC2 
singularities of Cohen-Macaulay type $t=2$ and 
dimension $2$ or $3$, i.e. we also allow nonisolated singularities 
in the Tjurina transform. 

\begin{theorem}{(Main theorem)}
  Let $X_\varepsilon$ be the Milnor fiber of an ICMC2 singularity
  $(X_0, 0) \subset (\CC^{n+2},0)$ of dimension $n=2$ or $3$ and
  Cohen-Macaulay type $t=2$ 
  given by a matrix $A \in \Mat(3,2;\CC\{\underline x\})$ and 
  its perturbation $A_\varepsilon$. 
  Let 
  \[ 
    L : X_\varepsilon \to \PP^1,\quad x\mapsto \Span A_\varepsilon^T(x).
  \]
  The Milnor fiber $X_\varepsilon$ is simply connected
  and the homology of $X_\varepsilon$ splits into 
  \[
    H_\bullet(X_\varepsilon) \cong G^\perp \oplus \ZZ.
  \]
  The cap product with $L^*(H^2(\PP^1))$ gives a perfect 
  pairing of the vertical cycles with $H^2(\PP^1) \cong \ZZ$. 
  If $n=3$, then $H_2(X_\varepsilon) \cong \ZZ$ consists of the 
  vertical cycles only.
  \label{thm:MainTheorem}
\end{theorem}

Since for any given example the Euler characteristic $\chi(X_\varepsilon)$ 
can be computed by e.g. the polar multiplicities \cite{NOT} of $(X_0,0)$, 
we obtain the following corollary.

\begin{corollary}
  Let $X_\varepsilon$ be the Milnor fiber of an 
  ICMC2 threefold singularity of Cohen-Macaulay type $t=2$ 
  The Betti numbers 
  of $X_\varepsilon$ 
  can be computed as 
  \begin{eqnarray*}
    b_0 &=&  1 \\
    b_2 &=&  1 \\
    b_3 &=&  - \chi(X_\varepsilon) + 2 \\
    b_k &=& 0 \quad \textnormal{ for } k \notin \{ 0, 2, 3 \}, 
  \end{eqnarray*}
\end{corollary}

\subsection{An example and outline of the proof}

To illustrate the ideas of the proof of theorem \ref{thm:MainTheorem}, we give an example 
of an ICMC2 threefold singularity with non-isolated singular locus in the 
Tjurina transform.

\vspace{0.5cm}
  Let $(X_0,0)\subset (\CC^5, 0)$ be given by the matrix
  \begin{align}
    \begin{pmatrix}
      v & x \\
      w & y \\
      -2xy & v^2 + w^2 +z^2 
    \end{pmatrix}
    \label{eqn:DefiningMatrixExample1}
  \end{align}
  and consider the smoothing obtained by perturbing the lower left
  entry with a constant $\delta$. 
  We denote the homogeneous coordinates of $\PP^1$ by $(s_1:s_2)$. 
  Then the equations for the Tjurina transform 
  $(Y_0,V) \subset (\CC^{5}\times \PP^1, \{0\}\times \PP^1)$ 
  and its deformation by $\delta$ 
  are 
  \begin{align}
    \begin{pmatrix}
      v & x \\
      w & y \\
      -2xy-\delta & v^2 + w^2 +z^2 
    \end{pmatrix}
    \cdot 
    \begin{pmatrix}
      s_1 \\ s_2
    \end{pmatrix}
    = 0.
    \label{eqn:DefiningEquationsTjMod}
  \end{align}

  Let us look at the first chart $\{s_1 \neq 0\}$. 
  We write $s = s_2/s_1$ for the corresponding standard affine coordinate.
  The equations from the first two rows read 
  \[
    v = -s\cdot x, \quad w = -s \cdot y.
  \]
  Substituting this in the equation from the last row, 
  we obtain a hypersurface
  \[
    h = s^3\cdot x^2 + s^3 \cdot y^2 - 2xy + s \cdot z^2, 
  \]
  which is perturbed by a constant $\delta$.
  We can interprent this as a quadratic form $Q_s$ in $(x,y,z)$ 
  parametrized by $s$ and write it in the standard matrix form:
  \[
    h = Q_s(x,y,z) = 
    \begin{pmatrix}
      x & y & z 
    \end{pmatrix} 
    \cdot 
    \begin{pmatrix}
      s^3 & -1 & 0 \\
      -1 & s^3 & 0 \\
      0 & 0 & s 
    \end{pmatrix}
    \cdot 
    \begin{pmatrix}
      x \\ y \\ z 
    \end{pmatrix}
    = \delta
  \]
  Any quadratic form should of course be diagonalized. To do this, 
  we introduce new coordinates 
  \[
    \begin{pmatrix}
      \tilde x \\ \tilde y \\ \tilde z
    \end{pmatrix}
    :=
    \begin{pmatrix}
      \frac{1}{\sqrt{2}} & \frac{-1}{\sqrt{2}} & 0 \\
      \frac{1}{\sqrt{2}} & \frac{1}{\sqrt{2}}  & 0 \\
      0 & 0 & 1 
    \end{pmatrix}
    \cdot 
    \begin{pmatrix}
      x \\ y \\ z 
    \end{pmatrix}, 
  \]
  in which our hypersurface equation takes the form 
  \begin{equation}
    h = Q_s(\tilde x, \tilde y, \tilde z) =
    \begin{pmatrix}
      \tilde x & \tilde{y} & \tilde{z} 
    \end{pmatrix}
    \cdot 
    \begin{pmatrix}
      s^3 + 1 & 0 & 0 \\
      0 & s^3 -1 & 0 \\
      0 & 0 & s
    \end{pmatrix}
    \cdot 
    \begin{pmatrix}
      \tilde x \\ \tilde{y} \\ \tilde{z} 
    \end{pmatrix}
    = \delta
    \label{eqn:HypersurfaceDiagonalizedExample1}
  \end{equation}
  This is a family of $A_1$-surface singularities, which degenerates as 
  $s$ approaches one of the seven values 
  \[
    s \in \sqrt[6]{1} \cup \{0\}.
  \]
  Now it is clear that in this chart the Tjurina transform $Y_0$ is singular along 
  the whole exceptional set $V$, the $s$-axis in this chart.
  
  Let $L : \CC^5 \times \PP^1 \to \PP^1$ be the standard projection, i.e. in this 
  chart the projection to the $s$-axis. 
  If we restrict $h$ to a general transversal slice to $V$ given by the hypersurface 
  $\{ L = c \}$ for a general $c \in \CC$, we obtain the 
  \textit{transversal singularity}, denoted by $Y_0^\pitchfork$
  \[
    h|_{\{L=c\}}  = (c^3+1) \tilde{x}^2 + (c^3-1) \tilde{y}^2 + c \tilde{z}^2 = \delta
  \]
  and a smoothing induced by the perturbation with the constant $\delta$. 
  This transversal singularity is isolated and of type $A_1$.

  For $\delta \neq 0$ we see a vanishing cycle $[\sigma]$ in the Milnor fiber  
  \[
    Y_\delta^\pitchfork = \{ h = \delta \} \cap \{ L = c \}.
  \]
  of the transversal singularity. It lives in the second homology 
  group $H_2(Y_\delta^\pitchfork)$ and can be represented by a $2$-sphere. 
  This is a candidate for further contributions of the second homology 
  group of 
  \[
    Y_\delta \subset \CC^5 \times \PP^1,
  \]
  the fiber over $\delta$ in the given deformation, and hence for the Milnor 
  fiber $X_\varepsilon$ of $(X_0,0)$.
  Whether or not $[\sigma]$ is nonzero as an element of $H_2(Y_\delta)$ 
  depends on the inclusion 
  \[
    Y_\delta^{\pitchfork} \subset Y_\delta.
  \]
  To shed some light on this question, let us observe the behaviour 
  close to the degeneracy points 
  \[
    K := \left\{ (\tilde{x},\tilde{y},\tilde{z},s) : \tilde{x} = \tilde{y} = \tilde{z} = 0, 
    s \in \sqrt[6]{1} \cup \{0\} \right\}.
  \]
  The analytic type of the singularity $h$ at either of these points is what 
  D. Siersma calls the $D_\infty$-singularity, a.k.a. the
  \textit{Whitney umbrella}.
  For any $p\in K$ we can choose a Milnor ball $B = B(p)$ for the singularity 
  of $h$ around $p$ and a value $c \in \CC$ for the transversal singularity 
  sufficiently close to $s(p)$ such that the intersection $B$ with the 
  hyperplane $\{L=c\}$ is nonempty. 
  D. Siersma shows in \cite{Siersma83}, proposition 3.8: 

  \vspace{0.5cm}
  \noindent
  For the $D_\infty$ singularity of dimension $n$
  the pair of Milnor fibers $(Y_\delta \cap B, Y_\delta^\pitchfork\cap B)$ 
  is homotopy equivalent to the pair of spheres
  \[
    (S^{n},S^{n-1}), 
  \]
  where $S^{n-1} \hookrightarrow S^{n}$ is the standard equatorial embedding.
  \vspace{0.5cm}

  Let $W\subset \CC$ be the complement of some small discs around the special 
  points in $K \subset \CC$. Then for $\delta>0$ small enough 
  \begin{equation}
    L : Y_\delta \cap L^{-1}(W)) \to W
    \label{eqn:LFiberBundleOverOpenSet}
  \end{equation}
  is a fiber bundle with fiber $Y_\delta^\pitchfork$. This means we 
  can freely move the equator of all the vanishing cycles coming from 
  the seven $D_\infty$ points and connect all half spheres globally. 
  The affine part of $Y_\delta$ is therefore homotopic to a bouquet of 
  spheres:
  \begin{equation}
    Y_\delta \setminus \{ s_1 \neq 0\} \cong 
    \underbrace{S^3 \vee \cdots \vee S^3}_{2\cdot 7 -1 \textnormal{ times}},
    \label{eqn:YDeltaWithoutInfinityBouquetOfSpheres}
  \end{equation}
  with each of their equators being homologous to the vanishing cycle 
  $S^2$ of 
  any of the transversal Milnor fibers.

  To complete the picture, let us look at the other chart $\{ s_2 \neq 0 \}$.
  We denote the corresponding affine coordinate of $\PP^1$ by $t = s_1/s_2$. 
  Again the equations for the first two rows of the matrix allow us to substitute 
  in the equation of the third row and we obtain the perturbation of a 
  hypersurface equation: 
  \[
    h := v^2 + w^2 - 2t^3 \cdot vw + z^2 = \delta \cdot t.
  \]
  Regarding this as a quadratic form $Q_t$ in $(v,w,z)$ parametrized by $t$ 
  and diagonalizing as before, we obtain 
  \[
    \begin{pmatrix}
      \tilde{v} & \tilde{w} & \tilde{z} 
    \end{pmatrix} 
    \cdot 
    \begin{pmatrix}
      1+t^3 & 0 & 0 \\
      0 & 1-t^3 & 0 \\
      0 & 0 & 1
    \end{pmatrix}
    \cdot 
    \begin{pmatrix}
      \tilde{v} \\
      \tilde{w} \\
      \tilde{z}
    \end{pmatrix}
    = 
    \delta \cdot t.
  \]
  We do recover the six degeneracy values for $t$ of the quadratic form 
  at the six roots of unity. However, 
  $Q_t$ does not degenerate at the point $(0,\infty) \in \CC^{5} \times \PP^1$, 
  the origin in this chart. 
  Hence we can make an analytic change of coodinates around this 
  point such that the local equation $h$ for $Y_0$ at $(0,\infty)$ is just an 
  $A_\infty$ singularity:
  \[
    h = x^2 + y^2 + z^2 : (\CC^4,0) \to (\CC,0).
  \]
  But note that we do not perturb by a constant, but by $\delta \cdot t$. 
  This means, the transversal slice over $\{ t = 0 \}$ does not deform! 
  We have 
  \[
    Y^\pitchfork_\infty := Y_0 \cap \{ t = 0 \} = Y_\delta \cap \{ t= 0 \}.
  \]
  The set $Y^\pitchfork_\infty$ is what we call an \textit{axis of the deformation} 
  and its intersection with the exceptional set $V$ the 
  \textit{axis point} $(0,\infty)$. 
  
  Being a representative of the germ of an 
  isolated singularity, $Y^\pitchfork_\infty$ is a contractible fiber 
  in the family given by 
  $L : Y_\delta \to \PP^1$. 
  For the vanishing cycle $[\sigma]$ of the transversal Milnor fiber, 
  the equator of all the $3$-spheres generating the homology of 
  $Y_\delta \setminus \{s_1 \neq 0\}$, this gives one more opportunity to close. 
  Hence we obtain: 
  \[
    H_3 (Y_\delta) \cong \ZZ^{14}
  \]
  is freely generated by $3$-spheres.

  But $Y_\delta$ is not homotopic to a bouquet of spheres, as 
  one might think at this point. There is a nontrivial cycle in $H_2(Y_\delta)$ 
  sitting over $\PP^1$, which is constructed as follows.

  Recall (\ref{eqn:LFiberBundleOverOpenSet})
  that $Y_\delta \cap L^{-1}(W)$ had the structure of a fiber bundle 
  over $W$ by means of $L$ .
  The fiber $Y_\delta^\pitchfork$ is homotopic to $S^2$, while the base 
  $W$ has the homotopy type of a bouquet of $6$ circles. Obstruction theory 
  tells us, that up to homotopy there is a unique continous section 
  \[
    l : W \to Y_\delta \cap L^{-1}(W) 
  \]
  of $L$. Because over $\infty \in \PP^1$ we only glue in a contractible fiber, we 
  can certainly extend $l$ to $W \cup \{\infty\}$. 
  Let $B = \overline{ \PP^1 \setminus (W \cup \{\infty\})}$ be the closure of the 
  complement of $W \cup \{\infty\}$. 
  The image of $l$ defines a unique relative cycle $[l]$ in 
  $H_2(Y_\delta, L^{-1}(B))$, 
  whose boundary consists of seven circles in the links of the local Milnor fibers 
  of the $D_\infty$ points. 
  At every such point $p_i$ we can choose local coordinates, in which $Y_\delta$ 
  is in the normal form  
  \[
    s\cdot x^2 + y^2 + z^2 = \delta.
  \]
  We can extend $l$ by glueing 
  \[ 
    s \mapsto (s,x,y,z) = (s,0,0,\sqrt{\delta})
  \]
  to the respective part of the boundary $\partial [l]$ of $[l]$ and 
  obtain a global section $l' : \PP^1 \to Y_\delta$.
  Consider the long exact sequence of the pair $(Y_\delta,L^{-1}(B))$:
  \[
    \xymatrix{
      H_2(L^{-1}(B)) \ar[r]^\iota & 
      H_2(Y_\delta) \ar[r] &
      H_2(Y_\delta, L^{-1}(B)) \ar[r]^{\partial^*} & 
      H_1(L^{-1}(B)) 
    }
  \]
  The existence of the local extensions of $l$ tells us that 
  $\partial^*([l]) = 0$ and also $H_2(Y_\delta \cap L^{-1}(B) = 0$.
  Hence $H_2(Y_\delta)$ is freely generated by 
  $[l']$, the image of the fundamental class of $\PP^1$ under $l$. 
  It is evident that
  \[
    L^* : H^2(\PP^1) \to H^2(Y_\delta) = \Hom(H_2(Y_\delta),\ZZ)
  \]
  is an isomorphism.

  Because the deformation we started with was a smoothing of $(X_0,0)$, 
  the spaces $Y_\delta$ and $X_\delta$ are naturally isomorphic and we're done 
  with the determination of the homology groups of the Milnor fiber 
  of $(X_0,0)$.

\vspace{0.5cm}
We will now outline the proof of the main theorem \ref{thm:MainTheorem} using 
this example. It is widely inspired by the work of D. Siersma and M. Tibar on 
the vanishing topology of projective hypersurfaces \cite{ST15}, in the way 
we piece together the global picture from local computations and 
the role played by the axis point. 

\vspace{0.5cm}
\noindent
\textit{Step I: Study line singularities, which are local complete intersections.} 
In general the singular locus $V = \{0\} \times \PP^1$ 
of the Tjurina transform $Y_0 \subset \CC^5 \times \PP^1$ will consist 
of a Zariski open set $U$, over which the projection $L$ to $\PP^1$ induces the structure 
of a fiber bundle with fiber $Y_0^\pitchfork$, the transversal singularity. 
Its Milnor fiber $Y_\delta^\pitchfork$ is well defined up to diffeomorphism. 
This is done in section \ref{subsec:MilnorFiberProductCase}. 

Then we will treat the special points, i.e. the complement of $U$. In the above example
we saw that the vanishing cycle $[\sigma]$ of the transversal singularity became 
homologous to zero in the local Milnor fibers of the $D_\infty$ singularities. 
But in the general case of arbitrary 
line singularities which are complete intersections, 
there is no reason for this to hold. 
Consider for example the $F_1 A_3$ singularity from De Jongs list 
\cite{DeJong88}: 
\[ 
  f = xz^2 +y^2 z = z\cdot (xz + y^2).
\]
He shows that its Milnor fiber $F$ is homotopy equivalent to $S^1$. If 
we find such a singularity in the Tjurina transform of an ICMC2 surface singularity 
or a double suspension of it in the Tjurina transform of a threefold, then 
there are cycles of the transversal Milnor fiber $F^\pitchfork$, 
which are not homologous 
to zero in $F$.

It turns out that the important 
property we need, is the fact that any vanishing cycle of degree $(n-1)$ 
of the Milnor fiber $F$ of a complete intersection line singularity 
can be represented by a cycle in the transversal Milnor fiber $F^\pitchfork$.
This is done in section \ref{subsec:MilnorFiberSpecialPoint}, where we 
give a description of how the local Milnor fiber of those singularities 
is connected to its transversal Milnor fiber 
(corollary \ref{cor:TopologySecondBoundary} and theorem 
\ref{thm:ConnectivitySecondBoundary} for the threefolds and respectively 
\ref{cor:TopologySecondBoundarySurface} and \ref{thm:ConnectivitySecondBoundarySurface} 
for the surface case). 

\vspace{0.5cm}
\noindent
\textit{Step II: The role of the axis point.} 
In section \ref{subsec:GenericRankOnePerturbationAndAxisPoint} we show 
that for deformations of an ICMC2 singularity $(X_0,0)$ of dimension $n$ and its Tjurina 
transform $(Y_0,V)$ coming from a perturbation of the defining matrix $A$
with a general constant matrix $B$ of rank $1$, a 
\textit{generic rank 1 perturbation}, 
we always have an axis $Y^\pitchfork_\infty$ and an axis point $(0,\infty) \in V$.
For the fiber $Y_\delta$ of the Tjurina transform in such a deformation,  
the connectivity of local Milnor fibers $F$ of complete intersection line 
singularities with their transversal Milnor fibers $F^\pitchfork$ will 
imply that all homology of degree $n-1$ of $Y_\delta \setminus Y^\pitchfork_\infty$ 
is concentrated in the transversal Milnor fiber $Y_\delta^\pitchfork$. 
When glueing in the fiber $Y^\pitchfork_\infty$ of $L$ over $\infty$, all the cycles 
in $Y_\delta^\pitchfork$ collapse.

\vspace{0.5cm}
\noindent
\textit{Step III: Putting together the global picture. }
In the last part, section \ref{subsec:TheGlobalPicture},
we use Mayer-Vietoris arguments to compute the 
homology groups of $Y_\delta$ for a general rank $1$ perturbation 
(theorem \ref{thm:SecondHomologyGroupRank1Perturbation}): 
While the vanishing cycles of $Y_\delta^\pitchfork$ and hence also 
all $(n-1)$-cycles of the local Milnor fibers of the special points 
are homologous 
to zero in $Y_\delta$, adding $Y_\infty^\pitchfork$ to 
$Y_\delta \setminus Y_\infty^\pitchfork$ will also give 
rise to a new $2$-cycle sitting over $\PP^1$ in the sense of 
definition \ref{def:SittingOver}.
This finally leads to the proof of the main theorem \ref{thm:MainTheorem}, 
in which we pass from a general rank $1$ perturbation, for which neither 
$Y_\delta$ nor $X_\delta$ are necessarily smooth, to a smoothing of $(X_0,0)$.

\section{Topology of line singularities}

\begin{definition}
  A singularity $(Y_0,0) \subset (\CC^N,0)$ is called a \textit{line singularity}, 
  if the singular locus $V = \Sing(Y_0)$ is the germ of a line in $\CC^N$ at $0$. 
\end{definition}

\noindent
Curves cannot have line singularities -- unless they are a multiple line themselves. 
In this section, we will therefore always assume $n = \dim (Y_0,0) \geq 2$.

Let $(Y_0,0) \subset (\CC^N,0)$ be a line singularity, which is a complete intersection 
of codimension $d$ given by the equations $f_1= \dots = f_d = 0$. 
For line singularities there is in general no unique smoothing. 
But if we consider the defining equations $f_i$ as components of a 
map germ 
\[
  f : (\CC^N,0) \to (\CC^d,0),
\]
then for a chosen Milnor ball $B$ for $(Y_0,0)$, 
the preimage of a regular value $c \in \CC^d$ of $f$ on $B$ sufficiently close 
to $0$ is unique up to diffeomorphism. This is what we will refer to as 
the (local) Milnor fiber of the line singularity $(Y_0,0)$.

There is a well-known trick to reduce the problem to a Milnor fiber 
of one holomorphic function on a controlled ambient space, see e.g. \cite{Hamm72}.
Let $0 \in U \subset \CC^N$ be a neighborhood of the origin, on which 
all the $f_i$ are defined.
Consider the map 
\[
  \PP f : U \setminus Y_0 \to \PP^{d-1}, \quad x \mapsto (f_1(x):\dots:f_d(x))
\]
and choose a regular value $p\in \PP^{d-1}$ for $\PP f$.
After a change of coordinates of $\PP^{d-1}$, which corresponds to a new 
$\CC$-linear combination 
of the generators $f_i$, we can assume that $p = (1:0:\dots:0)$. Then the closure 
of its preimage in $U \subset \CC^N$ is given by 
\begin{equation}
  Y^* = \{ x \in U : f_2(x) = \dots = f_d(x) = 0 \}.
  \label{eqn:DefinitionYStar}
\end{equation}

\begin{lemma}
  The singular locus of $Y^*$ is contained in the singular locus of $Y$.
  \label{lem:SingularLocusYStar}
\end{lemma}

\begin{proof}
  (cf. \cite{Hamm72}, Lemma 1.1 or Lemma 2.2) 
  Outside $Y_0$ the space $Y^*$ is already smooth. If $Y^*$ had 
  a singular point $p\in Y_0$, this means that the jacobian of 
  $(f_2,\dots,f_d)$ does not have full rank at $p$. But then also the jacobian 
  of $(f_1,f_2,\dots,f_d)$ cannot have full rank and hence $p$ is 
  a singular point of $Y_0$ as well.
\end{proof}

We rename the first function $f_1$ by $f$. Without loss of generality we can 
assume, that the singular line $(V,0) = (\Sing(Y),0)$ is just the germ of the 
first coordinate 
axis of $\CC^N$. 
This will be the \textit{standard situation}, from which we will proceed for this section:
\begin{eqnarray}
  &f : (Y^*,0)\subset (\CC^N,0) \to (\CC,0),
  \label{eqn:StandardSituationAfterHypersurfaceReduction} \\
  &(\Sing(Y^*),0 ) = (\{x_2 = \dots = x_N = 0\},0).
  \label{eqn:SingularLocusYStar}
\end{eqnarray}

Since we are primarily interested topological questions about the singularity, we will use 
Whitney stratifications to provide the setup for applications of the 
first Thom isotopy lemma. 
We may assume that $Y^*$ admits a Whitney stratification by the strata 
\begin{equation}
  (Y^*\setminus Y_0,Y_0 \setminus V, V \setminus \{0\}, \{0\} )
  \label{eqn:WhitneyStratificationLocal}
\end{equation}
sufficiently close to the origin. The last stratum $\{0\}$ might however be 
optional.

\subsection{The polar curve}

Besides the Whitney stratification there is one more thing we need to take 
into account. Let 
\[
  L : \CC^N \to \CC, (x_1,\dots,x_N) \mapsto x_1 
\]
be the projection to the first coordinate axis. The polar locus of $f$ with 
respect to $L$ on $Y^*$ is defined as 
\begin{equation}
  \Gamma(f,L) = 
  \overline{\{x\in Y^*\setminus Y_0 : \D L(x),\D f(x) 
    \textnormal{ are linearly dependent in } \Omega^1_{Y^*}\}},
  \label{eqn:DefinitionPolarLocus}
\end{equation}
where $\overline{\cdot}$ denotes the closure. The polar locus can 
be very nasty. However for most choices of the projection $L$ we get a 
reasonable control over $\Gamma$.

\begin{lemma}
  For $a= (a_2,\dots,a_N)$ let 
  \[
    L_a : \CC^N \to \CC, \quad (x_1,\dots,x_N) \mapsto x_1 - \sum_{i=2}^N a_i \cdot x_i
  \]
  be the \textit{projection bent by} $a$. 
  There exists a dense set $\Omega \subset \CC^{N-1}$ of values 
  for $a$, such that for $a \in \Omega$ the polar locus 
  $\Gamma(f,L_a) \subset Y^*$ is either empty or an analytic curve, which is smooth 
  outside $Y_0$. 
  \label{lem:ChoiceProjectionPolarCurve}
\end{lemma}

\begin{proof}
  This is a Bertini-type theorem. Consider the following incidence space 
  \[
    N^*:= \{ (x,a) \in Y^* \times \CC^{N-1} : 
    \D L_a(x), \D f(x) 
    \textnormal{ are linearly dependent in } \Omega^1_{Y^*} \}.
  \]
  It comes along with the two natural projections
  \[
    \xymatrix{
      & 
      N^* \ar[dl]_{pr_1} \ar[dr]^{pr_2} & 
      \\
      Y^* & 
      & 
      \CC^{N-1}\\
    }
  \]
  Over $Y^*\setminus Y_0$ the function $f$ does have a full rank differential on 
  $TY^*$ and therefore $N^*\setminus pr_1^{-1}(Y_0)$ is a smooth manifold 
  of complex dimension 
  \[
    \dim N^* = \dim Y^* + (N-1) - (\dim Y^* -1 ) = N.
  \]
  Let $a$ be a regular value of the projection $pr_2$ restricted to 
  $N^* \setminus pr_1^{-1}(Y_0)$. Its preimage
  \[
    pr_2^{-1}(\{a\}) \subset Y^* \times \{a\}
  \]
  is either empty or an analytic curve, which is smooth outside $Y_0 \times \{a\}$. 
\end{proof}

We will in the following assume that $L_a$ has been chosen according to Lemma 
\ref{lem:ChoiceProjectionPolarCurve}. 
Then we readjust the coordinate system of $\CC^N$ in a way that $L_a = L = x_1$ 
is just the first coordinate function, i.e. the projection to the first axis.

\begin{corollary}
  Passing to a smaller representative of $Y_0$ if necessary, 
  we can furthermore assume that 
  the polar curve meets $Y_0$ only at points in $V$. 
  \label{cor:ChoiceProjectionPolarCurve}
\end{corollary}

\subsection{The choice of a Milnor ball}

Let $\rho : \CC^N \to \RR$ be the squared distance function from the origin 
and set $B_\varepsilon := \{ \rho \leq \varepsilon \}$.
For sufficiently small $\varepsilon > 0$ we may assume that 
\begin{itemize}
  \itemsep0cm
  \item $\rho$ is a Whitney stratified submersion on 
    $Y^* \cap B_\varepsilon$ with respect to the standard 
    stratification (\ref{eqn:WhitneyStratificationLocal}).
  \item (cf. \cite{Hamm72}, Korollar 3.2) the function 
    \[
      \arg f : Y^*\setminus Y_0 \to S^1
    \]
    has a differential which is linearly independent of 
    $\D \rho$ in $T^*Y^*$ along $B_\varepsilon \cap (Y^*\setminus Y_0)$.
  \item the function $f$ has no critical points on 
    $B_\varepsilon \cap Y^*$ away from $Y_0$. 
  \item the polar curve $\Gamma$ is either empty or 
    it intersects $B_\varepsilon \cap V$ at most at the origin.
\end{itemize}

\subsection{The Milnor fiber in the product case}
\label{subsec:MilnorFiberProductCase}

In this section we will treat the case that 
\[
  (Y^*\setminus Y_0, Y_0 \setminus V, V)
\]
is already a Whitney stratification of $Y^*$ at $0$ and the polar curve 
$\Gamma$ is empty. 

Thom's first isotopy lemma yields that $Y_0$ is a product over $V$: 
\begin{equation}
  (Y_0,0) \cong (Y_0^\pitchfork \times V,0),
  \label{eqn:ProductStructureSingularFiber}
\end{equation}
where $(Y_0^\pitchfork,0)$ is the germ of the \textit{transversal singularity}. 
This is an isolated singularity obtained from $Y_0$ by intersecting it with 
a hyperplane in general position, i.e. transversal to all strata at $0$. 
Lemma \ref{lem:ChoiceProjectionPolarCurve} and corollary 
\ref{cor:ChoiceProjectionPolarCurve} show us, how to choose the equation for 
such a hyperplane. 

We will show that the product structure (\ref{eqn:ProductStructureSingularFiber}) 
also holds for the Milnor fiber. To do so, it is more convenient to have a polydisc 
rather than a Milnor ball. 
Assume that the projection $L = x_1$ to the first axis is general and let 
\[
  q = \sum_{i=2}^N \overline x_i \cdot x_i : \CC^N \to \RR
\]
be the squared distance from $V$. 
According to \cite{TopStab76}, lemma 2.3, the map 
\[
  (L, q ) : Y_0 \setminus V \to V \times \RR
\]
is a submersion on a neighborhood $U$ of the origin. 
We may choose $\alpha, \beta \in \RR_{>0}$ small enough such that 
the polydisc 
\[
  \Delta_{\alpha \beta} := \{ q \leq \alpha^2\} \cap \{|L|\leq \beta\}
\]
is contained in $U$.

\begin{theorem}
  In the above setup for fixed $\alpha$ and $\beta$ there exists a $\delta >0$ 
  such that the map 
  \begin{equation}
    (f,L) : Y^* \cap \Delta_{\alpha \beta} \cap f^{-1}(D_\delta) \to 
    D_\delta \times D_\beta
    \label{eqn:FLSubmersion}
  \end{equation}
  is a fiber bundle away from $Y_0 = Y^* \cap \{ f=0 \}$. 
  \label{thm:MilnorFiberProductCase}
\end{theorem}

\begin{definition}
  The fiber of (\ref{eqn:FLSubmersion}) over a general point is called the 
  \textit{transversal Milnor fiber} and denoted by $F^\pitchfork$. 
\end{definition}

\noindent
Clearly for fixed $\delta >0$ we have 
$F \cong F^\pitchfork \times D_\beta$.

\begin{proof}{(of theorem \ref{thm:MilnorFiberProductCase})}
  Since $(L,q)$ was a submersion on $Y_0 \cap \Delta_{\alpha \beta}$, the 
  horizontal part of the boundary 
  \[
    \partial_h ( Y_0 \cap \Delta_{\alpha \beta} ) := 
    Y_0 \cap \{ q = \alpha^2\} \cap \{|L|\leq \beta\}
  \]
  is a fiber bundle over the closed disc $D_\beta$. Because it is compact, 
  this property is preserved under small perturbations of $f$. Hence we can 
  assume that 
  \[
    (f,L) : Y^* \cap f^{-1}(D_\delta) \cap \{ q=\alpha^2\} \cap L^{-1}(D_\beta)
    \to D_\delta \times D_\beta
  \]
  is a fiber bundle. 
  
  The absence of the polar curve assures that away from $Y_0$ we also find no critical 
  points of $(f,L)$ in the interior of $Y^* \cap \Delta_{\alpha \beta}$. 
  Therefore (\ref{eqn:FLSubmersion}) is a proper submersion away from $Y_0$ and 
  hence a fiber bundle by Ehresmann's fibration theorem.
\end{proof}

\subsection{The Milnor fiber at a special point}
\label{subsec:MilnorFiberSpecialPoint}

We now treat the general case, i.e. we have a Whitney stratification of 
$Y^* \subset \CC^N$ by the strata 
\[
  (Y^*\setminus Y_0, Y_0 \setminus V, V\setminus\{0\},\{0\})
\]
and a possibly nonempty polar curve $\Gamma \subset Y^*$ which meets $Y_0$ at $\{0\}$. 
By passing to smaller representatives if necessary, we can always reduce to 
this setup. 

Let $B$ be a Milnor ball for $Y_0$ at $0$. 
When we investigate the topology at the special points in the setting of the 
Tjurina-modification of an ICMC2 singularity, it is the part of the boundary 
$\Sigma = Y_0 \cap \partial B$
which is close to $V$, along which $Y_0$ connects to the remaining space. Therefore 
we will study mainly two objects in this section: The topology of the second boundary 
\[
  \partial_2 F \subset \partial F, 
\]
which is the part of the boundary of the Milnor fiber $F$ close to $V$, 
and the relative homology groups 
\[
  H_q(F,\partial_2 F),
\]
which determine how $F$ is connected to $\partial_2 F$. 
The precise definition of the second boundary $\partial_2 F$ is given below.

\subsubsection{The second boundary}
\label{subsec:TheSecondBoundary}

In this section we will denote the boundaries of the spaces in question by 
\[
  \Sigma^* := Y^* \cap \partial B, \quad \Sigma := Y_0 \cap \partial B, \quad 
  S := V \cap \partial B
\]
Along the points of $S$ we find the product situation of the preceeding 
section for $Y_0$. Thus theorem \ref{thm:MilnorFiberProductCase} is applicable along the 
whole circle. 
However we do need the slight modification to change $L$ to 
\[
  \tilde L : B \to \CC, \quad x \mapsto \sqrt{\rho(x)} \cdot \exp( \sqrt{-1}\cdot \arg L ),
\]
with $\rho = |L|^2 + q$ the squared distance from the origin.
The function $\tilde L$ is not holomorphic, 
but approximates $L$ as a differentiable 
function close to $S$. 
Repeating the arguments in the setup and proof of theorem \ref{thm:MilnorFiberProductCase} 
along the compact manifold $S$ we obtain:

\begin{corollary}
  There exist $\alpha>0$ and $\delta>0$ sufficiently small with respect to $\alpha$ 
  such that 
  \begin{equation}
    (f, \arg L ) : \Sigma^* \cap \{ q \leq \alpha^2\} \cap f^{-1}(D_\delta) 
    \to D_\delta \times S^1
    \label{eqn:FullFibrationSecondBoundary}
  \end{equation}
  is a smooth fiber bundle away from $\{ f = 0 \}$. 
  \label{cor:FibrationSecondBoundary}
\end{corollary}

It is easy to see that the fiber of this fiber bundle is canonically diffeomorphic 
to the transversal Milnor fiber $F^\pitchfork$. 

\begin{definition}
  For $\alpha$ and $\delta$ as in corollary \ref{cor:FibrationSecondBoundary} the 
  space 
  \[
    \partial_2 F := \Sigma^* \cap \{ q \leq \alpha^2 \} \cap \{f = \delta\}
  \]
  is called the \textit{second boundary} of the Milnor fiber $F$
  and the monodromy from the fibration 
  \begin{equation}
    \arg L : \partial_2 F \to S^1
    \label{eqn:FibrationSecondBoundary}
  \end{equation}
  the \textit{vertical monodromy}.
\end{definition}

The topology of $\partial_2 F$ is completely determined by the 
topology of $F^\pitchfork$ and the monodromy operator of the Wang sequence of 
(\ref{eqn:FibrationSecondBoundary}). 
The transversal Milnor fiber $F^\pitchfork$ comes from an ICIS of dimension 
$n-1$, so it is $(n-2)$-connected. For $n \geq 3$ the Wang sequence splits into two parts
\begin{equation}
  \xymatrix{
    0 \ar[r] &
    H_n(\partial_2 F) \ar[r] & 
    H_{n-1}( F^\pitchfork ) \ar[r]^{\mathbf T_{n-1}- \mathbf 1} &
    H_{n-1}( F^\pitchfork ) \ar[r] & 
    H_{n-1}(\partial_2 F) \ar[r] &
    0
  }
  \label{eqn:WangSequenceVerticalMonodromyI}
\end{equation}
and
\begin{equation}
  \xymatrix{
    0 \ar[r] &
    H_1(\partial_2 F ) \ar[r] &
    H_0(F^\pitchfork) \ar[r]^{\mathbf T_0-\mathbf 1} &
    H_0(F^\pitchfork) \ar[r] &
    H_0(\partial_2 F ) \ar[r] &
    0
  }
  \label{eqn:WangSequenceVerticalMonodromyII}
\end{equation}
where $\mathbf T_\bullet$ is the monodromy operator of (\ref{eqn:FibrationSecondBoundary}).
Clearly, $\mathbf T_0 -\mathbf 1$ in 
(\ref{eqn:WangSequenceVerticalMonodromyII}) 
is the zero map. Thus we proved the following.

\begin{corollary}
  Let $n= \dim(Y_0,0) \geq 3$. 
  The homology groups of $\partial_2 F$ have the following properties:
  \begin{enumerate}
    \item $H_n(\partial_2 F)$ is a free subgroup of $H_{n-1}(F^\pitchfork)$.
    \item Every cycle in $H_{n-1}(\partial_2 F)$ can be represented 
      by a cycle in $H_{n-1}(F^\pitchfork)$. 
    \item $H_1(\partial_2 F)$ is free abelian of rank $1$ and generated by a 
      section of $\arg L$.
    \item $\partial_2 F$ is connected.
    \item All other homology groups are zero.
  \end{enumerate}
  \label{cor:TopologySecondBoundary}
\end{corollary}

If $n = \dim(Y_0,0) = 2$, the terms $H_{n-1}(\partial_2 F)$ from 
(\ref{eqn:WangSequenceVerticalMonodromyI}) and $H_1(\partial_2 F)$ in 
(\ref{eqn:WangSequenceVerticalMonodromyII}) come together. But the kernel 
of $\mathbf T_0 - \mathbf 1$ in is still 
free of rank $1$ and hence there is a (non-canonical) splitting
\begin{equation}
  H_1(\partial_2 F ) \cong H_1' \oplus \ZZ = \coker (T_1-\mathbf 1) \oplus \ker (T_0-\mathbf 1).
  \label{eqn:SplittingH1SecondBoundarySurfaceCase}
\end{equation}
We call $H_1' = \coker T_1- \mathbf 1$ the \textit{transversal} or 
\textit{horizontal} and the other summand 
$\ZZ = \ker T_0 - \mathbf 1$ the \textit{vertical} cycles of the second boundary 
$\partial_2 F$. 

\begin{corollary}
  The homology groups of the second boundary $\partial_2 F$ of the Milnor 
  fiber $F$ of a complete intersection line singularity $(Y_0,0)$ of dimension 
  $2$ have the following properties:
  \begin{enumerate}
    \item $H_2(\partial_2 F)$ is a free subgroup of $H_1(F^\pitchfork)$.
    \item Every cycle in $H_{1}(\partial_2 F)$ can be represented 
      by a transversal cycle in $H_{1}(F^\pitchfork)$. 
    \item $H_1(\partial_2 F)$ 
      splits into transversal and vertical cycles 
      (\ref{eqn:SplittingH1SecondBoundarySurfaceCase}) and a generator 
      of the latter is given by the fundamental class of a section of $\arg L$.
    \item $\partial_2 F$ is connected.
    \item All other homology groups are zero.
  \end{enumerate}
  \label{cor:TopologySecondBoundarySurface}
\end{corollary}

\subsubsection{Connectivity of the Milnor fiber with the second boundary}

Having described the topology of the second boundary we now turn to the question, 
how it connects with the Milnor fiber. We will first treat the case 
$n = \dim(Y_0,0) \geq 3$ and modify the arguments for the surface case in the next 
section.

\begin{theorem}
  Let $n=\dim Y_0 \geq 3$. Then we have 
  \begin{equation}
    H_q(F, \partial_2 F ) \cong 
    \begin{cases}
      0 & 2 < q < n\\
      H_{q-1}(\partial_2 F) & q = 2 \\
      0 & 0\leq q \leq 1
    \end{cases}
    \label{eqn:RelativeHomologyGroupsMilnorFiberSecondBoundary}
  \end{equation}
  where the isomorphisms are induced from the long exact sequence 
  of the pair of spaces $(F, \partial_2 F)$.
  \label{thm:ConnectivitySecondBoundary}
\end{theorem}

The proof of theorem \ref{thm:ConnectivitySecondBoundary} follows closely the ideas 
of Dirk Siersma in his paper \cite{Siersma91}. He proved it in the case of 
hypersurfaces, with possibly even more complicated singular locus, as the corollary of 
lemma 3.8, his ``second variation sequence''
\footnote{There is a typo in \cite{Siersma91}: The third case in the mentioned
corollary is $2<q\leq n-1$.}.
It picks up the idea of the original fibration by Milnor 
\begin{equation}
  \arg f : \Sigma^* \setminus \Sigma \to S^1,
  \label{eqn:OriginalMilnorFibration}
\end{equation}
where as before $\Sigma^* = Y^* \cap \partial B$ and $\Sigma$ is the boundary of $Y_0$. 
Hamm shows in \cite{Hamm72}, Satz 1.6, that this is a $C^\infty$-fiber bundle 
with open fibers.
Moreover he proves that for $\delta > 0$ sufficiently small, 
(\ref{eqn:OriginalMilnorFibration}) is 
in fact fiberwise diffeomorphic to 
\begin{equation}
  \frac{f}{\delta}: \{|f| = \delta \} \cap Y^* \cap \overset{\circ}{B} \to S^1.
  \label{eqn:MilnorFibrationTwo}
\end{equation}
The proof proceeds by construction of an outward pointing vector field on 
$Y^* \setminus Y_0$, whose flow takes $\{|f|= \delta \} \cap Y^* \cap B$ fiberwise 
onto $\Sigma^*\setminus \{|f|\leq\delta\}$. For two chosen single fibers we can then 
establish an isomorphism. 

Unlike in the case of an ICIS it is not so easy to see that if we pass to the 
closure in $B$, we still get a fibration.

\begin{lemma}
  For sufficiently small $\delta >0$ the map 
  \begin{equation}
    \frac{f}{\delta} : \{|f| = \delta \} \cap Y^* \cap B \to S^1
    \label{eqn:MilnorFibrationClosedFibers}
  \end{equation}
  is a $C^\infty$ fiber bundle 
  with closed fibers 
  \[
    F = \{f=\delta\} \cap Y^* \cap B.
  \]
  \label{lem:MilnorFibrationClosedFibers}
\end{lemma}

\begin{proof}
  By choice of the Milnor ball, there are no critical points of $f$ on 
  $(Y^*\setminus Y_0 ) \cap B$.
  Hence we only have to check that $f/|\delta|$ is a submersion at the boundary 
  \[ 
    \{ |f| = \delta \} \cap \Sigma^*.
  \]
  This can be achieved by using first the curve selection lemma to show 
  that $f$ has no critical points on $\Sigma^* \setminus \Sigma$ on a neighborhood
  $U$ of $\Sigma$, and then the compactness of $\Sigma$: For sufficiently small 
  $\delta$ the set $\{|f| \leq \delta \} \cap \Sigma^*$ will be contained in $U$.
\end{proof}

\noindent
To create the setup to prove theorem \ref{thm:ConnectivitySecondBoundary}, 
we first choose $\alpha > 0$ such that
\begin{itemize}
  \itemsep0cm
  \item all requirements of corollary \ref{cor:FibrationSecondBoundary} are 
    fulfilled, so that we will have a fibration of the second boundary.
  \item the space 
    \[ 
      N_\alpha := \Sigma^* \cap \{ q \leq \alpha^2 \}
    \]
    has $S = V\cap \Sigma$ as a strong deformation retract in $\Sigma^*$. 
\end{itemize}
After that we choose $\delta >0$ sufficiently small with respect to $\alpha$ such that 
\begin{itemize}
  \itemsep0cm
  \item again the assumptions of corollary \ref{cor:FibrationSecondBoundary} 
    are met.
  \item lemma \ref{lem:MilnorFibrationClosedFibers} holds and we get a Milnor fibration 
    by $f$.
  \item we have $\Sigma$ as a strong deformation retract of the space
    \[ 
      \Sigma_{\leq \delta} := \Sigma^* \cap \{ |f| \leq \delta \} 
    \]
    and the retraction takes the subset 
    $\partial N_\alpha\cap \Sigma_{\leq \delta}$ into itself.
\end{itemize}
This last space now decomposes as 
\[
  \Sigma_{\leq \delta} = 
  ( \Sigma_{\leq \delta} \cap \overline{(\Sigma^* \setminus N_\alpha)} ) 
  \cup 
  (\Sigma_{\leq \delta} \cap N_\alpha ) 
  =: 
  T_1
  \cup 
  T_2 .
\]
The attentive reader may recognize $T_2$ from corollary 
\ref{cor:FibrationSecondBoundary}. The other part $T_1$ has a natural structure 
as a trivial disc bundle over $\Sigma$ as in the case of isolated singularities, 
since $\Sigma \setminus N_\alpha$ was compact
and smooth. 

Now we can according to Hamm's computations decompose the space $\Sigma^*$ as 
\begin{eqnarray*}
  \Sigma^* &\cong& (\Sigma^*\cap \{|f|\leq \delta\}) \cup 
  	 (\Sigma^* \cap \{|f| \geq \delta\} )\\
	 &\cong& (\Sigma_{\leq \delta}) \cup 
	 ( Y^* \cap B \cap f^{-1}(D_\delta) ),
\end{eqnarray*}
where the second part is a smooth fiber bundle over the circle by 
lemma \ref{lem:MilnorFibrationClosedFibers}. 

\begin{proof}{(of theorem \ref{thm:ConnectivitySecondBoundary})}
  Consider the triple of spaces 
  $(\Sigma^*, F \cup T_2, T_2)$.
  We have the following isomorphisms for the relative homology groups.
  \begin{eqnarray}
    H_q(\Sigma^*, F\cup T_2) &\cong& 
    \label{eqn:IdentificationThirdPairI}
    H_q(\Sigma^*, F \cup T_2 \cup T_1) \\
    &\cong& H_q( Y^* \cap B \cap \{ |f| = \delta \}, F \cup (\Sigma^* \cap \{ |f| = \delta\}) ) 
    \label{eqn:IdentificationThirdPairII}\\
    &\cong& H_q( F \times [0,1], \partial( F \times [0,1] ) )
    \label{eqn:IdentificationThirdPairIII} \\
    &\cong& H_{q-1}(F,\partial F) \otimes H_1(I,\partial I) = H_{q-1} ( F, \partial F )
    \label{eqn:IdentificationThirdPairIV}
  \end{eqnarray}
  for $q>0$ and $H_0(\Sigma^*,F\cup T_2) = 0$.
  The first line (\ref{eqn:IdentificationThirdPairI}) holds, because $T_1$ retracts 
  onto the part of the boundary of $F$ outside $N_\alpha$. By excision we get 
  (\ref{eqn:IdentificationThirdPairII}) and (\ref{eqn:IdentificationThirdPairIII}) 
  comes from the fibration (lemma \ref{lem:MilnorFibrationClosedFibers}). Then 
  (\ref{eqn:IdentificationThirdPairIV}) comes from the K\"unneth formula.
  \begin{eqnarray}
    H_q(\Sigma^*, T_2 ) & \cong & H_q(\Sigma^*, N_\alpha \cap \Sigma ) \\
    &\cong& H_q(\Sigma^*, S ),
    \label{eqn:IdentificationSecondPair}
  \end{eqnarray}
  because by assumption $T_2$ retracts onto $N_\alpha \cap \Sigma$ which 
  in turn retracts onto $S = V \cap \Sigma$.
  Finally by excision we deduce
  \begin{eqnarray}
    H_q(F \cup T_2, T_2) & \cong & H_q ( F, \partial_2 F ).
    \label{eqn:IdentificationFirstPair}
  \end{eqnarray}

  \noindent
  With these identifications the long exact sequence from the triple reads
  \begin{equation}
    \xymatrix{
      \cdots \ar[r] &
      H_{q+1}(\Sigma^*, F\cup T_2 ) \ar[r] \ar[d]^\cong&
      H_q(F\cup T_2, T_2 ) \ar[r] \ar[d]^\cong& 
      H_q(\Sigma^*, T_2) \ar[r] \ar[d]^\cong& 
      \cdots \\ 
      \cdots \ar[r] & 
      H_q(F, \partial F ) \ar[r] & 
      H_q(F, \partial_2 F ) \ar[r] & 
      H_q(\Sigma^*, S ) \ar[r] &
      \cdots
    }
    \label{eqn:SecondVariationSequence}
  \end{equation}
  Recall that according to Andreotti and Fraenkel's proof of the Lefschetz hyperplane 
  theorem (see e.g. \cite{Mil63}), the relative homology groups 
  $H_q (F,\partial F )$ vanish for $q < n$. Thus we find isomorphisms
  \begin{equation}
    H_q(F,\partial_2 F) \cong H_q (\Sigma^*, S) \qquad \textnormal{ for } q < n.
    \label{eqn:IdentificationRelativeHomologyGroups}
  \end{equation}
  To determine the connectivity of the pair $(F,\partial_2 F)$ we are therefore left 
  with the computation of the relative homology groups $H_q(\Sigma^*,S)$. 

  The rest of the proof will split into three cases.
  In any of these, we will show that from the long exact sequence in homology of the pair 
  $(\Sigma^*, S )$ we get 
  \begin{equation}
    H_q(\Sigma^*, S ) \cong
    \begin{cases}
      0 & \textnormal{ for } 2 < q < n \\
      H_{q-1}(S) = \ZZ & \textnormal{ for } 0 < q \leq 2
    \end{cases}
    \label{eqn:ConnectivitySigmaStarS}
  \end{equation}

  \vspace{0.5cm}
  \noindent
  \textit{Case I: } $Y^*$ is smooth. \\
  This has been done by Dirk Siersma in 
  \cite{Siersma91}. The pair $(\Sigma^*,S)$ is just $(S^{2n-1},S^1)$ with the 
  usual equatorial embedding.
  Clearly (\ref{eqn:ConnectivitySigmaStarS}) holds and 
  (\ref{eqn:RelativeHomologyGroupsMilnorFiberSecondBoundary}) follows for 
  the case $0 \leq q < n, q \neq 2$. For $q=2$ consider the following commutative 
  diagram
  \begin{equation}
    \xymatrix{
      H_2(F,\partial_2 F ) \ar[d] \ar[r]^\cong &
      H_2(F\cup T_2,T_2 ) \ar[d] \ar[r]^\cong &
      H_2(\Sigma^*, T_2 ) \ar[d] \ar[r]^\cong &
      H_2(\Sigma^*, S ) \ar[d]^\cong \\
      H_{1}(\partial_2 F ) \ar[r]^\cong &
      H_{1}(T_2) \ar[r]^\cong &
      H_{1}(T_2) \ar[r]^\cong &
      H_{1}(S) 
    }
    \label{eqn:CommutativeDiagramSecondBoundaryLowDegrees}
  \end{equation}
  All horizontal maps are isomorphisms. In the lower row they are induced 
  by the inclusion $\partial_2 F \hookrightarrow T^2$ and the retraction
  of $T_2$ onto $S$.
  The vertical map on the right clearly is an isomorphism, too.
  This finishes the proof in case I.

  \vspace{0.5cm}
  \noindent
  \textit{Case II: } $Y^*$ has an isolated singular point at the origin. \\ 
  In this case $\Sigma^*$ is a smooth compact manifold. Let $F^*$ be the Milnor 
  fiber of the isolated complete intersection singularity $(Y^*,0)$. The dimension 
  of $F^*$ is $n+1$ and according to Hamm it is homotopic to a bouquet 
  of $(n+1)$-dimensional spheres. 
  The Lefschetz hyperplane theorem asserts that one can obtain $F^*$ from 
  $\Sigma^*$ by attaching cells of dimension $\geq n+1$. Then clearly 
  $\Sigma^*$ must be $(n-1)$-connected and (\ref{eqn:ConnectivitySigmaStarS}) 
  follows from the long exact sequence of the pair $(\Sigma^*,S)$.
  The proof is finished with the same arguments as in case I.

  \vspace{0.5cm}
  \noindent
  \textit{Case III:} $Y^*$ is also singular along $V$. \\
  Here $S$ the singular part of the 
  boundary $\Sigma^*$ of $Y^*$. For $\alpha$ sufficiently small, it is homotopic to the 
  pair $(\Sigma^*, N_\alpha)$. Let again $F^*$ be the Milnor fiber in a smoothing 
  of $Y^*$ and consider the triple $(F^*,\partial F^*,\partial_2 F^*)$. 
  By excision we clearly have isomorphisms
  \[
    H_q(\Sigma^*, S ) \cong 
    H_q (\Sigma^*,N_\alpha ) \cong H_q ( \partial F^*, \partial_2 F^* ) 
  \]
  for all $q$. The long exact sequence for the triple reads 
  \begin{equation*}
    \xymatrix{
      \cdots \ar[r] & 
      H_{q+1}(F^*,\partial F^* ) \ar[r] &
      H_q(\partial F^*, \partial_2 F^* ) \ar[r] & 
      H_q(F^*,\partial_2 F^* ) \ar[r] &
      \cdots 
    }
  \end{equation*}
  and for $q+1 < n+1 = \dim F^*$, the terms $H_{q+1}(F^*,\partial F^*)$ vanish.
  Thus for all $0<q<n$ we have isomorphisms
  \begin{equation}
    H_q( F, \partial_2 F ) \cong 
    H_q( \Sigma^*, S ) \cong H_q( \partial F^*, \partial_2 F^*) \cong 
    H_q( F^*, \partial_2 F^*).
    \label{eqn:InductionStepGoingUp}
  \end{equation}
  The claim now follows by induction on the codimension of $Y^*$.
  For $0\leq q<n, q \neq 2$ the right hand term of (\ref{eqn:InductionStepGoingUp}) is 
  zero and in case $q=2$ we can extend the diagram 
  (\ref{eqn:CommutativeDiagramSecondBoundaryLowDegrees}) by one more column to obtain
  \begin{equation}
    \xymatrix{
      H_2(F,\partial_2 F ) \ar[d] \ar[r]^\cong &
      \cdots \ar[r]^\cong &
      H_2(\Sigma^*, S ) \ar[d] \ar[r]^\cong &
      H_2(F^*, \partial_2 F^*) \ar[d]^ \cong \\
      H_{1}(\partial_2 F ) \ar[r]^\cong &
      \cdots \ar[r]^\cong &
      H_{1}(S) \ar[r]^\cong&
      H_{1}(\partial_2 F^*)
    }
  \end{equation}

\end{proof}

\subsubsection{The surface case} 

We already saw in section \ref{subsec:TheSecondBoundary}, corollary 
\ref{cor:TopologySecondBoundarySurface}, that surfaces need special treatment. 
The reason for this is that the horizontal and the vertical cycles of 
the second boundary $\partial_2 F$ do not live in distinct homology groups 
anymore. In view of its applications for ICMC2 singularities in the next section, 
we will formulate a different connectivity result for the pair $(F,\partial_2 F)$ 
in the case $n=2$. 

\begin{theorem}
  Let $(Y_0,0) \subset (\CC^N,0)$ be a complete intersection line singularity 
  of dimension $n=2$.
  Recall the decomposition 
  \[
    H_1(\partial_2 F) \cong H_1' \oplus \ZZ
  \]
  into horizontal and vertical cycles 
  (\ref{eqn:SplittingH1SecondBoundarySurfaceCase}) 
  for the second boundary $\partial_2 F$ of the Milnor fiber $F$ of $Y_0$. 
  With these identifications, the natural map $\iota_1 : H_1(\partial_2 F ) \to H_1(F)$ 
  is surjective and factors via 
  \begin{equation}
    \xymatrix{
      H_1' \oplus \ZZ \ar@{->>}[rr]^{\iota_1} \ar[dr] & 
      & 
      H_1(F) \\
      & 
      H_1' \ar[ur]
      & 
      \\
    }
    \label{eqn:FactorizationInclusionSecondBoundaryMilnorfiberSurface}
  \end{equation}

  \label{thm:ConnectivitySecondBoundarySurface}
\end{theorem}
  
\noindent
In other words: The vertical cycles are homologous to zero in $F$, while 
every remaining $1$-cycle of $F$ comes from a cycle in $\partial_2 F$.

\begin{proof}
  We can literally copy the setup and the beginning of the proof of 
  theorem \ref{thm:ConnectivitySecondBoundary} up to the point, where 
  we deduce the isomorphisms 
  (\ref{eqn:IdentificationRelativeHomologyGroups}).
  From here the proof of theorem \ref{thm:ConnectivitySecondBoundarySurface}
  becomes an investigation of the part 
  \[
    \xymatrix{
      H_2(F,\partial_2 F) \ar[r] &
      H_1(\partial_2 F ) \ar[r] & 
      H_1( F ) \ar[r] & 
      H_1(F,\partial_2 F) \ar[r] & 
      0
    }
  \]
  of the long exact sequence from the pair $(F,\partial_2 F)$.

  Let $l : S^1 \to \partial_2 F$ be a section of $\arg L$ representing 
  the homology class $[l]$ of the generator of the vertical cycles in 
  $H_1(\partial_2 F)$. Consider the commutative diagram
  \begin{equation}
    \xymatrix{
      H_2(F,\partial_2 F ) \ar[d] \ar[r]^\cong &
      H_2(F\cup T_2,T_2 ) \ar[d] \ar[r]^\cong &
      H_2(\Sigma^*, T_2 ) \ar[d] \ar[r]^\cong &
      H_2(\Sigma^*, S ) \ar[d]^\delta \\
      H_{1}(\partial_2 F ) \ar[r]^\kappa &
      H_{1}(T_2) \ar[r]^\cong &
      H_{1}(T_2) \ar[r]^\cong &
      H_{1}(S) 
    }
    \label{eqn:CommutativeDiagramSecondBoundaryLowDegreesSurface}
  \end{equation}
  where the column maps are the natural ones from the corresponding pairs of 
  spaces. 
  Contrary to the higher dimensions, the map $\kappa$ coming 
  from the inclusion $\partial_2 F \hookrightarrow T_2$ is not necessarily 
  an isomorphism anymore. But clearly it maps $[l]$ into the homology class 
  of the generator of $H_1(S) \cong \ZZ$. 
  
  The factorization 
  (\ref{eqn:FactorizationInclusionSecondBoundaryMilnorfiberSurface}) would 
  follow from $\delta$ on the right being surjective.
  Surjecitivity of $\iota_1$ in 
  (\ref{eqn:FactorizationInclusionSecondBoundaryMilnorfiberSurface}) 
  directly follows from $H_1(F,\partial_2 F)$ being zero.

  \vspace{0.5cm} 
  \noindent
  \textit{Case I:} $Y^*$ is smooth (cf. \cite{Siersma91}). \\
  The pair $(\Sigma^*,S)$ is nothing but a pair of spheres $(S^5,S^1)$ with the standard 
  equatorial embedding. Clearly $\delta$ in 
  (\ref{eqn:CommutativeDiagramSecondBoundaryLowDegreesSurface}) is surjective and from 
  (\ref{eqn:IdentificationRelativeHomologyGroups}) we get
  \[
    H_1(F,\partial_2 F) \cong H_1(S^5,S^1) = 0.
  \]

  \vspace{0.5cm} 
  \noindent
  \textit{Case II:} $Y^*$ has an isolated singularity at the origin. \\
  The Milnor fiber $F^*$ of $Y^*$ is a bouquet of spheres of dimension $3$ 
  and the pair $(F^*,\partial F^*)$ is $2$-connected. 
  We consider a smoothing of $(Y^*,0)$ compatible with the constructions made 
  for $(Y_0,0)$. 
  The term 
  \[
    H_1(F,\partial_2 F) \cong H_1(\Sigma^*,S) \cong H_1(\partial F^*, \partial_2 F^*)
  \]
  appears in the long exact sequence of the triple $(F^*,\partial F^*, \partial_2 F^*)$:
  \[
    \xymatrix{
      \cdots \ar[r] & 
      H_2(F^*, \partial F^*) \ar[r] & 
      H_1(\partial F^*, \partial_2 F^* ) \ar[r] & 
      H_1(F^*,\partial_2 F^*) \ar[r] & 
      \cdots
    }
  \]
  Both terms on the left and the right are zero 
  due to the connectivity of $(F^*,\partial F^*)$ and 
  the long exact sequence of the pair $(F^*,\partial_2 F^*)$.
  Consequently $H_1(F,\partial_2 F)$ also is. 
  
  The space $\Sigma^* = \partial F^*$ is $1$-connected, for if it wasn't,
  according to the 
  connectivity with its boundary, $F^*$ couldn't be a bouquet 
  of spheres. This shows surjectivity of $\delta$ in 
  (\ref{eqn:CommutativeDiagramSecondBoundaryLowDegreesSurface}).
  
  \vspace{0.5cm} 
  \noindent
  \textit{Case III:} $Y^*$ is singular along $V$. \\
  For the surjectivity of $\delta$ in 
  (\ref{eqn:CommutativeDiagramSecondBoundaryLowDegreesSurface}), we apply 
  the same argument as in the higher dimensional case. 
  From the long exact sequence of the triple 
  $(F^*,\partial F^*, \partial_2 F^*)$ and the connectivity of $(F^*,\partial F^*)$ 
  and theorem \ref{thm:ConnectivitySecondBoundarySurface}
  we deduce surjectivity of the natural map 
  \[
    H_2(\Sigma^*,S) \cong H_2(\partial F^*,\partial_2 F^*) \to 
    H_2( F^*,\partial_2 F^*) \cong H_1(\partial_2 F).
  \]
  With this we can 
  extend the diagram (\ref{eqn:CommutativeDiagramSecondBoundaryLowDegreesSurface}) 
  to the right by the column 
  \begin{equation}
    \xymatrix{
      H_2(F,\partial_2 F ) \ar[d] \ar[r]^\cong &
      \cdots \ar[r]^\cong &
      H_2(\Sigma^*, S ) \ar[d]^\delta \ar@{->>}[r] &
      H_2(F^*,\partial_2 F^*) \ar[d]^\cong\\
      H_{1}(\partial_2 F ) \ar[r]^\kappa &
      \cdots \ar[r]^\cong &
      H_{1}(S) \ar[r]^\cong &
      H_1(\partial_2 F^*) \\
    }
  \end{equation}
  Also from the connectivity of $(F^*,\partial_2 F^*)$ we get
  \[
    H_1(F,\partial_2 F ) \cong H_1(\Sigma^*,S) \cong H_1(F^*,\partial_2 F^*) = 0.
  \]
\end{proof}

\section{Application to ICMC2 singularities of type $2$}

Let $(X_0,0)\subset (\CC^{n+2},0)$ be an ICMC2 singularity of Cohen-Macaulay type 
$t=2$ described by the matrix 
\[
  A = 
  \begin{pmatrix}
    a_{1,1} & a_{1,2} \\
    a_{2,1} & a_{2,2} \\
    a_{3,1} & a_{3,2}
  \end{pmatrix}.
\]
so that the Tjurina transform 
$(Y_0,\{0\} \times \PP^1) \subset (\CC^{n+2} \times \PP^1, \{0\}\times \PP^1)$ 
is given by the equations 
\[
  \begin{pmatrix}
    f_1 \\ f_2 \\ f_3
  \end{pmatrix} 
  := 
  \begin{pmatrix}
    a_{1,1} & a_{1,2} \\
    a_{2,1} & a_{2,2} \\
    a_{3,1} & a_{3,2}
  \end{pmatrix}
  \cdot
  \begin{pmatrix}
    s_1 \\ s_2
  \end{pmatrix}
  = 0.
\]
Assume that $Y_0$ is singular along the exceptional set 
$V= \{0\} \times \PP^1$.
We may choose a Whitney stratification for $Y_0$ by strata 
\[
  (Y_0\setminus V, V\setminus \{p_1,\dots,p_N\}, \{ p_1,\dots,p_N\}).
\]
The first part of this section is devoted to creating a setup, 
in which the conditions for the methods and results of section 2 are met. 

\vspace{0.5cm}
First we construct the space $Y^*$ globally by the same arguments. Let 
$X_0 \subset U \subset \CC^{n+2}$ be a representative of $(X_0,0)$ in some open 
neighborhood $U$ of the origin and $Y_0 \subset U \times \PP^1$ its 
Tjurina transform. Consider 
\[
  \PP f : U\times \PP^1 \setminus Y_0 \to \PP^2, \quad 
  (x,s) \mapsto (f_1(x,s):f_2(x,s):f_3(x,s)).
\]
This is a well defined map although the $f_i$ are not functions. 
Choose a regular value $z\in \PP^2$ and define 
\[
  Y^* := \overline{ \PP f^{-1}(\{z\})} \subset U \times \PP^1.
\]
After a change of coordinates of $\PP^2$ sending $z$ to $(0:0:1)$, 
which naturally translates to row operations on $A$, we can assume that 
$Y^*$ is given by the equations $f_1 = f_2 = 0$ and that 
\[
  Y_0 = \{ f = 0 \} \cap Y^* \subset Y^*
\]
is the zero locus of $f:= f_3 \in H^0(U\times \PP^1,\mathcal O(1))$. 

\vspace{0.5cm}
Next we define the polar curve. Let 
$  L : Y^* \subset \CC^{n+2} \times \PP^1 \to \PP^1$
be the projection to $\PP^1$ and 
$z\in \PP^1$ a regular value of $L$ on $Y_0\setminus V$. We may, after a change 
of coordinates, which corresponds to a canonical column operation on $A$,
assume that $z=(0:1) = \infty$. In the chart $\{s_1 \neq 0\}$ we can 
play the same game as in lemma \ref{lem:ChoiceProjectionPolarCurve} in the 
whole chart at once to obtain a bending of $L$, which is sufficiently 
general for our needs. Any chosen bending in this chart 
will not alter the fiber of $L$ over 
$\infty$.

Observe that on the overlap $\{ s_1 \neq 0 \} \cap \{s_2 \neq 0\}$ the polar loci 
of the functions $f/s_1$ and $f/s_2 = f/s_1 \cdot s_1/s_2$ with respect 
to $L$ coincide. We can express $L$ as $s_2/s_1$. Then, because 
\[
  \D \frac{f}{s_1} = \D \left(\frac{f}{s_2} \cdot \frac{s_2}{s_1}\right) 
  = \frac{s_2}{s_1} \cdot \D \frac{f}{s_2} + \frac{f}{s_2} \cdot \D \frac{s_2}{s_1}
\]
clearly
\begin{eqnarray*}
  \Gamma &=&  \overline{\left\{ x \in Y^*\setminus Y_0 : 
    \D \frac{f}{s_2}(x) \textnormal{ and } \D \frac{s_2}{s_1} (x) 
    \textnormal{ are linearly dependent in } \Omega^1_{Y^*} \right\} } \\ 
  &=& \overline{ \left\{ x \in Y^* \setminus Y_0 : 
    \D \frac{f}{s_1}(x) \textnormal{ and } \D \frac{s_2}{s_1} (x) 
    \textnormal{ are linearly dependend in } \Omega^1_{Y^*} \right\} }. \\ 
\end{eqnarray*}
After possibly repeating the bending process of $L$ on the other chart, we have 
a well defined global polar curve $\Gamma \subset Y^*$, which is smooth 
outside $Y_0$ and meets $Y_0$ only at finitely many points along $V$. 
We add those points to the zero-dimensional stratum of the Whitney stratification of 
$Y_0$.

\subsection{The generic rank $1$ perturbation and the axis}
\label{subsec:GenericRankOnePerturbationAndAxisPoint}

Because $f = a_{3,1} \cdot s_1 + a_{3,2} \cdot s_2$ is a section of $\mathcal O(1)$ 
and not a function on $Y^*$, we can not globally perturb by a constant, but we 
have to choose another section 
$b = b_0 \cdot s_1 + b_1 \cdot s_2 \in H^0(\CC^{n+2}\times \PP^1, \mathcal O(1))$ and 
consider 
\[
  f - \delta \cdot b = 0
\]
in $\CC^{n+2} \times \PP^1 \times \CC$.
Thus there will always be one point in $V$, the zero locus of $b$, 
at which we will perturb the local 
equation of $f$ by zero. This point is called the \textit{axis point} of the 
deformation. It is unavoidable, but we can choose its position by the parameters 
$(b_1:b_2)$. 

Let us assume that after a change of coordinates the point 
$(0,\infty):= (0,(0:1)) \in V$ is not in the 
stratum $\{p_1,\dots,p_N\}$ of $Y_0$ and consider the deformation, which has 
$(0,\infty)$ as the axis point.
For the original ICMC2 singularity $(X_0,0)$ this means we consider the deformation
given by the perturbation
\begin{equation}
  \begin{pmatrix}
    a_{1,1} & a_{1,2} \\
    a_{2,1} & a_{2,2} \\
    a_{3,1} & a_{3,2}
  \end{pmatrix} - 
  \delta \cdot 
  \begin{pmatrix}
    0 & 0 \\
    0 & 0 \\
    1 & 0 \\
  \end{pmatrix}
  \label{eqn:GenericRank1Perturbation}
\end{equation}
This gives the equations for the total space $Y \subset \CC^{n+2} \times \PP^1 \times \CC$ 
of the deformation of $Y_0$ in the obvious way.

Note, that due to the generality assumptions in the choices of $Y^*$ and the axis 
point this is a \textit{generic rank 1 perturbation}. Every perturbation of $A$ by 
a constant matrix $B$ of rank $1$ can be brought to this form using row- and column 
operations on $A-\delta\cdot B$.

In the chart $\{s_1 \neq 0\}$ we now have a deformation of $Y_0$ given by the 
perturbation of 
\begin{equation}
  \frac{f}{s_1} : Y^* \setminus \{s_1 \neq 0\} \to \CC
  \label{eqn:GenericRank1PerturbationOffAxisPoint}
\end{equation}
by $\delta$. 
On the other hand at the axis point $(0,\infty)$ we find 
\begin{equation}
  \frac{f}{s_2} : (Y^*,(0,\infty)) \to (\CC,0)
  \label{eqn:GenericRank1PerturbationAtAxisPoint}
\end{equation}
perturbed by $\delta \cdot s$, where $s = s_1/s_2$ is the local coordinate 
of $\PP^1$ at $\infty$.

\subsection{$Y_\delta$ at the axis point} 

By assumption the axis point $(0,\infty)$ of the generic rank $1$ deformation was 
in general position along $V$. This means if we let 
$g = \frac{f}{s_2}$ be the local equation (\ref{eqn:GenericRank1PerturbationAtAxisPoint}) 
for $Y_0$ in $Y^*$ at the axis point, 
we find ourselves in the setup of 
theorem \ref{thm:MilnorFiberProductCase}.

Let $s = s_1/s_2$ and $x_1,\dots,x_{n+2}$ be local coordinates in this chart such that 
the point $(0,\infty)$ is the origin and choose
$\alpha, \beta>0$ as in theorem \ref{thm:MilnorFiberProductCase}.
Then for $\delta$ small enough the map 
\[
  G:=(g,s) : Y^* \cap \Delta_{\alpha \beta} \cap g^{-1}(D_\delta) \to 
  D_\delta \times D_\beta
\]
is a fiber bundle away from $Y_0 = G^{-1}(\{0\} \times D_\beta)$. 

The Milnor fiber of $g$ at $0$ is the preimage of a line $\{\delta\} \times D_\beta$ 
under this map. It inherits its product structure from the fibration by $s$.
To obtain the deformed fiber $Y_\delta$ of the generic rank $1$ deformation of $Y_0$ at 
$(0,\infty)$ from , we have to take a bent line 
\[
  W = \{ (\delta \cdot y, y) : y \in D_\beta \} \subset D_\delta \times D_\beta.
\]
Now $Y_\delta \cap \Delta_{\alpha \beta} = G^{-1}(W)$.
We deduce the following lemma:

\begin{lemma}
  Let $g : Y^* \to \CC$ be the local equation for $Y_0$ at the axis point 
  $(0,\infty) \in V$, 
  $s$ a local coordinate for $V$ at $(0,\infty)$ with $s(0,\infty) = 0$ and 
  $\Delta_{\alpha \beta}$ a chosen polydisc in the sense of 
  theorem \ref{thm:MilnorFiberProductCase}.
  Then for $\delta>0$ sufficiently small with respect to $\alpha$ and $\beta$, 
  the space
  \[
    Y_\delta \cap \Delta_{\alpha \beta} := Y^* \cap \Delta_{\alpha \beta} \cap 
    \{g=s\cdot \delta\}
  \]
  is the fiber over $\delta$ of the generic rank $1$ perturbation 
  (\ref{eqn:GenericRank1PerturbationAtAxisPoint}) close to the axis point.
  The map 
  \[ 
    L : Y_\delta \cap \Delta_{\alpha \beta} \to \PP^1
  \]
  is a fibration over a punctured neighborhood $D^\times_\beta$ of $\infty \in \PP^1$.
  The central fiber 
  \[
    Y^\pitchfork_\infty := Y_\delta \cap \Delta_{\alpha \beta} \cap \{L = \infty\} = 
    Y_0 \cap \Delta_{\alpha \beta} \cap \{L = \infty\},
  \]
  however, does not change as we pass from $Y_0$ to $Y_\delta$.
  Consequently $Y_\delta$ may retain a singular point at $(0,\infty)$. If this 
  happens, it is at most an ICIS.
  \label{lem:YdeltaAtAxisPoint}
\end{lemma}

\begin{definition}
  The space $Y_\infty^\pitchfork$ is called the axis of the deformation.
\end{definition}

\begin{corollary}
  The space $Y_\delta \cap \Delta_{\alpha \beta}$ as in lemma 
  \ref{lem:YdeltaAtAxisPoint} is contractible. 
  \label{cor:YdeltaContractible}
\end{corollary}

\begin{proof}
  The central fiber $Y^\pitchfork_\infty$ is a euclidean neighborhood retract of 
  some open neighborhood $U$ in $Y_\delta \cap \Delta_{\alpha \beta}$. Clearly 
  the fiber bundle $(Y_\delta \cap \Delta_{\alpha \beta}) \setminus Y^\pitchfork_\infty$ 
  can be retracted onto $U$ and successively onto $Y^\pitchfork_\infty$. Being 
  the representative of a germ of an isolated singularity in a Milnor ball, 
  $Y^\pitchfork_\infty$ is contractible. 
  Concatenation of these two contractions establishes the claim.
\end{proof}

\subsection{The global picture in the generic rank $1$ perturbation} 
\label{subsec:TheGlobalPicture}

After we already have a description of what happens at the axis point $(0,\infty)$ 
in a generic rank $1$ peturbation, 
let us now compute the topology of $Y_\delta$ in the other chart. 
To create a global setup first choose Milnor balls $B_i$ of radius $\varepsilon$ 
for all special points 
$\{p_1,\dots,p_N\}$ of $Y_0$. 
Let $B'_i$ be a Milnor ball of radius $\varepsilon/2$ around $p_i$ and set 
\[
  B = \bigcup_{i=1}^N B_i, \qquad B' = \bigcup_{i=1}^{N} B'_i.
\]
We now choose $\alpha>0$ sufficiently small such that 
\begin{itemize}
  \itemsep0cm
  \item all the local theory (theorem \ref{thm:ConnectivitySecondBoundary}
    for threefolds or theorem \ref{thm:ConnectivitySecondBoundarySurface} in the 
    surface case) 
    works at the special points $p_i$, 
  \item theorem \ref{thm:MilnorFiberProductCase} holds along the set 
    \[
      V' := V \setminus (B' \cup \{(0,\infty)\}),
    \]
    i.e. for $\delta>0$ small enough the map 
    \begin{equation}
      L : Y_\delta \cap \{q \leq \alpha^2\} \cap L^{-1}(V') \to V'
      \label{eqn:LInducingFiberBundleOverVprime}
    \end{equation}
    is a fiber bundle over $V'$ with fiber $F^\pitchfork$.
\end{itemize}
Note that for the last requirement we can use lemma 
\ref{lem:YdeltaAtAxisPoint} to achieve this behaviour in a neighborhood of the axis 
point. After that we're left with a compact subset of $V$, along which the existence 
of a global minimal $\alpha >0$ can certainly be assured.

Now we choose $\delta>0$ small enough with respect to all prior choices 
such that all the local theory developed above works at once along all points of the 
compact set $V$. 

We can now piece together the topology of $Y_\delta$ from the topology of 
the several known patches. 
We regard the axis point $(0,\infty)$ as a further special point $p_0$
in the Whitney stratification of $Y_0$. 
Let $\Delta$ be the chosen polydisk around $p_0$ and set 
\begin{equation}
  U:= Y_\delta \cap (B \cup \Delta), 
  \qquad W:= Y_\delta \cap \{q \leq \alpha^2 \} \cap L^{-1}( V' ).
  \label{eqn:ChoiceUandW}
\end{equation}
Furthermore let 
$\partial_2 F_i$ be the second boundary of the local Milnor fiber 
of $(Y_0,p_i)$ at $p_i$ for 
$i>0$. In case $i=0$, i.e. at the axis point, we just set 
\[
  \partial_2 F_0:= 
  Y_\delta \cap \{q \leq \alpha^2 \} \cap L^{-1}(\partial D_\beta),
\]
where $D_\beta$ was the chosen disc around $\infty \in \PP^1$.
We easily verify that the inclusion 
\[
  \partial_2 F_i \hookrightarrow (U\cap W)_i
\]
induces a homology equivalence 
\begin{equation}
  H_q(U\cap W) \cong \bigoplus_{i=0}^N H_q( (U \cap W)_i ) \cong
  \bigoplus_{i=0}^N H_q(\partial_2 F_i) 
  \label{eqn:HomologyEquivalenceIntersectionSecondBoundaries}
\end{equation}
where $(U \cap W)_i$ is the component of $U\cap W$ close to $p_i$.
For $q=1$ and $i>0$ let $[l_i]$ be the generator of $H_1(\partial_2 F_i)$ 
-- respectively the 
generator of the vertical part in case $n=2$ --
represented by a section 
$l_i : S^1 \to \partial_2 F_i$ of $\arg L$ in 
(\ref{eqn:FibrationSecondBoundary}), cf. corollary
\ref{cor:TopologySecondBoundary} and \ref{cor:TopologySecondBoundarySurface}.

The homology groups of $W$ itself are determined by its structure as a fiber bundle 
over $V'$ (\ref{eqn:LInducingFiberBundleOverVprime}).
Because $V'$ has the homotopy type of a finite bouquet of circles around the 
points $p_1,\dots,p_N$, 
we can basically repeat the arguments leading to corollary 
\ref{cor:TopologySecondBoundary} and \ref{cor:TopologySecondBoundarySurface}.
In particular we can assume
\begin{equation}
  H_0(W) = 0, \quad
  H_1(W)= 
  \begin{cases}
    \ZZ^N & \textnormal{ if } n = 3 \\
    H_1' \oplus \ZZ^N & \textnormal{ if } n = 2
  \end{cases}
  \label{eqn:HomologyGroupsWLowDegrees}
\end{equation}
where $H_1'$ is the quotient of $H_1(Y_\delta^\pitchfork)$ by the monodromies 
around all loops in the base $V'$ 
and in both cases $\ZZ^N$ is generated by the $[l_i]$. 
We can view the latter as sections of the generators of $H_1(V')$.

\begin{theorem}
  Let $Y_\delta$ be the fiber over $\delta \neq 0$ in 
  the genereric rank $1$ perturbation of 
  the Tjurina transform $(Y_0,V) \subset (\CC^{n+2}\times \PP^1, \{0\} \times \PP^1)$ 
  of an ICMC2 threefold singularity 
  $(X_0, 0) \subset (\CC^{n+2},0)$ of dimension $n=2$ or $3$ and
  Cohen-Macaulay type $t=2$. 
  Let 
  $ L : Y_\delta \to \PP^1 $ 
  be the projection to $\PP^1$
  and $G \subset H^\bullet(Y_\delta)$ the image of 
  $L^* : H^\bullet(\PP^1) \to H^\bullet(Y_\delta)$.
  Then $Y_\delta$ is simply connected
  and the homology of $Y_\delta$ splits into 
  \[
    H_\bullet(Y_\delta) \cong G^\perp \oplus \ZZ,
  \]
  where 
  $G^\perp = \{ [\sigma] \in H_\bullet(Y_\delta) : 
  g \cap [\sigma] = 0 \quad \forall g \in G \}$
  are the horizontal cycles of $Y_\delta$.
  The cap product with $L^*(H^2(\PP^1))$ gives a perfect 
  pairing of the vertical cycles $H_\bullet(Y_\delta)/G^\perp = \ZZ$
  with $H^2(\PP^1)$.
  If $n=3$, then $H_2(X_\varepsilon) \cong \ZZ$ consists of the 
  vertical cycles only.
  \label{thm:SecondHomologyGroupRank1Perturbation}
\end{theorem}

\begin{proof}
  Consider the Mayer-Vietoris sequence for $Y_\delta$ for the choice 
  (\ref{eqn:ChoiceUandW}) of the two patches 
  $U$ and $W$. First of all, the tail gives a short exact sequence
  \[
    \xymatrix{ 
      0 \ar[r] &
      H_0(U \cap W ) \ar[r] & 
      H_0(U) \oplus H_0(W) \ar[r] & 
      H_0(Y_\delta) \ar[r] & 
      0
    }
  \]
  and $Y_\delta$ is clearly connected. The first homology 
  group $H_1(Y_\delta)$ appears in the exact sequence 
  \begin{equation}
    \xymatrix{
      H_1(U \cap W ) \ar[r]^{\iota_1} &
      H_1(U) \oplus H_1(W) \ar[r] &
      H_1(Y_\delta) \ar[r] &
      0
    }.
    \label{eqn:MayerVietorisYDeltaH1}
  \end{equation}

  We proceed with the proof for the case $n=3$.
  From theorem \ref{thm:ConnectivitySecondBoundary} we know that $H_1(U) = 0$.
  On the generators chosen above, the map $\iota_1$ to the second 
  summand is given by the matrix
  \begin{equation}
    \begin{pmatrix}
      1 & 1 & 0 & \cdots & 0 \\
      1 & 0 & 1 & \ddots & \vdots \\
      \vdots & \vdots & \ddots & \ddots & 0 \\
      1 & 0 & \cdots & 0 & 1
    \end{pmatrix}
    \label{eqn:RepresentingMatrixIota1}
  \end{equation}
  and therefore clearly surjective. Thus $H_1(Y_\delta) = 0$ and 
  $Y_\delta$ is simply connected.

  Proceeding along the Mayer-Vietoris sequence to the left, we see in  
  (\ref{eqn:MayerVietorisYDeltaH2}) 
  that $H_2(Y_\delta)$ must be nonzero, because clearly the kernel of 
  (\ref{eqn:RepresentingMatrixIota1}) is free of rank $1$.
  \begin{equation}
    \xymatrix{
      H_2(U) \oplus H_2(W) \ar[r]^{\kappa_2} &
      H_2(Y_\delta) \ar[r]^{\partial_2} &
      H_1(U \cap W ) \ar[r]^{\iota_1} &
      H_1(U) \oplus H_1(W)
    }
    \label{eqn:MayerVietorisYDeltaH2}
  \end{equation}
  But $\kappa_2$ is in fact the zero map. To see this observe the following. 
  Every homology class $[\sigma] \in H_2(U)$ can be represented as a sum 
  of $2$-cycles in the boundaries 
  \[
    \sigma = \sum_{i=1}^N \sigma_i, \quad [\sigma_i] \in H_2(\partial_2 F_i)
  \]
  as a consequence of theorem \ref{thm:ConnectivitySecondBoundary}. 
  Corollary \ref{cor:TopologySecondBoundary} then tells us that $[\sigma_i]$
  even comes from a cycle in a transversal Milnor fiber 
  $[\sigma_i] \in H_2(F^\pitchfork_i)$ 
  close to $p_i$.
  The same holds for any $[\sigma] \in H_2(W)$ and any other chosen transversal 
  Milnor fiber over a point in $V'$. 

  Mapping any $[\sigma] \in H_2(U) \oplus H_3(W)$ into $H_2(Y_\delta)$ makes 
  it therefore homologous to a cycle in a transversal Milnor fiber 
  arbitrary close to $Y^\pitchfork_\infty$, 
  the fiber of $L$ over the axis point. Here it collapses, because 
  $Y_\delta \cap \Delta$ was contractible by corollary 
  \ref{cor:YdeltaContractible}. 

  Consequently $H_2(Y_\delta) = \ker \iota_1$. 
  We construct a generator for $H_2(Y_\delta)$ as follows. 
  Over $V' \cup \{\infty\}$ there exists a continous section 
  \[
    l : V' \to Y_\delta \cap L^{-1}(V')
  \]
  of $L$, because $L$ gives $Y_\delta \cap L^{-1}(V')$ the structure of a fiber bundle 
  with $1$-connected fiber $Y_\delta^\pitchfork$ over a base, which is homotopic to 
  a bouquet of $1$-spheres. 
  We can extend $l$ over $\infty$, because we only glue in a contractible fiber.
  Let $D = \overline{\PP^1 \setminus (V'\cup \{\infty\})}$ be the closure of the 
  complement of the domain of definition of $l$. 
  Then the fundamental class of the image of $l$ defines a unique relative cycle 
  \[
    [l] \in H_2(Y_\delta, L^{-1}(D)).
  \]
  Consider the following commutative diagram 
  \begin{equation}
    \xymatrix{
      H_2(L^{-1}(D)) \ar[r] &
      H_2(Y_\delta ) \ar[r]^\pi \ar[d]^{L_*} &
      H_2(Y_\delta, L^{-1}(D)) \ar[d]^{L_*} \ar[r] & 
      H_1(L^{-1}(D)) \\
      &
      H_2(\PP^1) \ar[r]^\cong & 
      H_2(\PP^1, D ) &
      \\
    }
    \label{eqn:SittingOverRelativeSequences}
  \end{equation}
  The image of $[l]$ in $H_1(L^{-1}(D))$ is zero by theorem 
  \ref{thm:ConnectivitySecondBoundary}: At each special point $p_i$ the component 
  of the boundary of $[l]$ in the local Milnor fiber is homologous to the 
  generator of $H_1(\partial_2 F) \cong H_2(F,\partial_2 F)$.
  On the other hand the map on the left into $H_2(Y_\delta)$ is the zero map 
  by the previous arguments: All $2$-cycles of the local Milnor fibers become 
  homologous to zero in $Y_\delta$. A generator $[\sigma]$ of $H_2(Y_\delta)$ is therefore 
  given as a preimage of $[l]$ under $\pi$. 

  The map $L_*$ on the right is an isomorphism and hence on the left $L_*$ maps 
  $[\sigma]$ to the fundamental class of $\PP^1$. 
  This finishes the proof for the threefolds.

  \vspace{0.5cm}
  If $n=2$, we need to modify the arguments above. First we show surjectivity 
  of $\iota_1$ in (\ref{eqn:MayerVietorisYDeltaH1}). 
  Recall that we 
  can split
  \[
    H_1(U\cap W) \cong \bigoplus_{i=0}^N H_1(\partial_2 F_i) \cong 
    \bigoplus_{i=0}^N \left(H_1'(\partial_2 F_i)\oplus \ZZ \right)
  \]
  into its horizontal and vertical part,
  where $H_1'(\partial_2 F_i)$ is the cokernel of $H_1(F^\pitchfork)$ by the 
  vertical monodromy at $p_i$.

  We can restrict the first component of $\iota_1$ mapping into 
  $H_1(U) = \bigoplus_{i=1}^N H'_1(F_i)$ to the summand 
  $\bigoplus_{i=1}^N H_1(\partial_2 F_i)$ and the second component of 
  $\iota_1$ mapping into $H_1(W)$ to $H_1'(\partial_2 F_0)\oplus \ZZ^{N+1}$. 
  Both restrictions 
  themselves are surjective by theorem \ref{thm:ConnectivitySecondBoundarySurface} 
  and (\ref{eqn:HomologyGroupsWLowDegrees}) 
  and hence also $\iota_1$ is. 

  On the vertical cycles $\ZZ^N$ of $H_1(U\cap W)$ the map $\iota_1$ takes again 
  the same form (\ref{eqn:RepresentingMatrixIota1}) and 
  consequently we can choose a splitting 
  \[
    H_2(Y_\delta) = H_2'(Y_\delta) \oplus \ZZ
  \]
  of the second homology group of $Y_\delta$ with the second summand 
  mapping to the kernel of $\iota_1$ on the vertical cycles. 
  We can construct a generator $[\sigma]$ of the quotient 
  $H_2(Y_\delta)/H_2'(Y_\delta) = \ZZ$ similar to the threefold case. 
  Start with a continous section 
  \[
    l : V' \cup \{\infty\} \to Y_\delta
  \]
  of $L : Y_\delta \to \PP^1$. For surfaces the relative homology 
  class $[l] \in H_2(Y_\delta,L^{-1}(D))$ is not unique, but depends 
  on the choice of $l$. Nevertheless, its preimage 
  $[\sigma] \in H_2(Y_\delta)$ under the map $\pi$ in 
  (\ref{eqn:SittingOverRelativeSequences}) generates the quotient 
  $H_2(Y_\delta)/H_2'(Y_\delta) = \ZZ$ and the composite map 
  $L_* \circ \pi$ is an isomorphism when restricted to the second summand 
  of the splitting $H_2(Y_\delta) = H_2'(Y_\delta) \oplus \ZZ$. 

  Hence again $[\sigma]$ is mapped to the fundamental class of $\PP^1$ by 
  $L_*$. All other cycles in $H_2(Y_\delta)$ can be represented sitting 
  in the preimage of discs or paths in $\PP^1$ and are therefore mapped to 
  zero by $L_*$. This finishes the proof for $n=2$.
\end{proof}

We can now prove the main theorem of this paper. 
\begin{proof}{(of theorem \ref{thm:MainTheorem})}
  Consider a deformation of $(X_0,0)$ with two parameters 
  $(\delta,\varepsilon)$, where the 
  first one $\delta$ is for a generic rank $1$ perturbation and the second one 
  $\varepsilon$ is for a smoothing. For the Tjurina transform $Y_{\delta,0}$ over 
  $X_{\delta,0}$, the fiber over $(\delta,0)$ for $\delta \neq 0$ small enough, 
  the homology groups are described by theorem 
  \ref{thm:SecondHomologyGroupRank1Perturbation}. However according to 
  lemma \ref{lem:YdeltaAtAxisPoint}, there might still be an ICIS of $Y_\delta$ 
  at the axis point. 

  In case $Y_{\delta,0}$ is smooth, its diffeomorphism type does not change as we pass to 
  a smooth fiber $Y_{\delta,\varepsilon}$ for $\delta,\varepsilon \neq 0$. 
  If it was not, its topology changes at most at the axis point $(0,\infty)$, 
  where it is the smoothing of an ICIS. 

  This means that, in the notation above, the local Milnor fiber 
  $Y_{\delta,\varepsilon} \cap \Delta$ of 
  $(Y_{\delta,0},(0,\infty))$ is $3$-connected. Hence all $1$- and $2$-cycles 
  appearing in the proof of theorem \ref{thm:SecondHomologyGroupRank1Perturbation}
  close to $Y^\pitchfork_\infty$ (i.e. representable by cycles in 
  $Y_{\delta,\varepsilon} \cap \Delta$)
  still become homologous to zero 
  in $Y_{\delta,\varepsilon}$ and 
  we can literally repeat all the arguments. 
  The theorem then follows from the isomorphism $Y_{\delta,\varepsilon} \cong 
  X_{\delta,\varepsilon}$.
\end{proof}

\vspace{0.5cm}
In this article we focused on ICMC2 singularities 
$(X_0,0)$ of Cohen-Macaulay type $2$, i.e. isolated 
determinantal singularities 
of type $(3,2,2)$.
The reason is, that 
in this case the worst one can get in the Tjurina transform $(Y_0,V)$
are line singularities. 
We saw that cycles from $(Y_0,V)$ get passed on to the Milnor fiber $X_\varepsilon$ 
and are then sitting over the homology of $\PP^1$ by means of the map $L$ 
associated to the deformed matrix. Computations for explicit examples as e.g. 
\cite{FKZ15}, example 3.5, yield that these phenomena can also be observed for 
ICMC2 singularities defined by bigger matrices. 

Consider as another example the threefold singularity $(X_0,0) \subset (\CC^5,0)$ 
defined by a generic embedding
\[
  A : \CC^5 \hookrightarrow \Mat(5,4;\CC)
\]
of a $5$-dimensional subspace into $\Mat(5,4;\CC)$. 
The Tjurina transform now decomposes as 
\[
  Y_0 = \overline{X_0} \cup (\{0\} \times \PP^3) \subset \CC^5 \times \PP^3,
\]
where $\overline{X_0}$ is the strict transform of $X_0$ and 
$\{0\}\times \PP^3$ is an additional component. The locus 
\[
  S = \overline{X_0} \cap (\{0\} \times \PP^3),
\]
where they meet, is a smooth projective hypersurface of degree $5$, 
so we encounter ``plane singularities'' in the Tjurina transform 
in the sense that the singular locus itself has dimension two!

Nevetheless the induced families 
in the Tjurina transform coming from deformations of $(X_0,0)$ are flat. 
Experimental computations show that the fiber $Y_\delta$ over $\delta \neq 0$ 
for a generic rank $2$ perturbation is already smooth and hence 
diffeomorphic to the Milnor fiber $X_\varepsilon$. The axis of such a deformation 
is a whole projective line $H \subset \PP^3$ and the fiber 
$Y_0\cap L^{-1}(H) = Y_\delta \cap L^{-1}(H)$ of $L$ sitting over it. 
This means that also the fundamental class of $H \cong \PP^1$ is 
passed on to $X_\varepsilon$ and then sitting over the corresponding cycle in 
$\PP^3$. Yet to develop a complete description of the topology of $X_\varepsilon$ 
in the spirit of this paper, we would need to deal with singular loci of 
dimension $2$ and their interplay with the topology of $S$ and the axis --  
a task, which is far more evolved than what has been done in this article. 

Apparently we gave a description in a special case 
of a more general phenomenon which yet remains 
to be explored: The vertical and horizontal vanishing cycles of isolated 
determinantal singularities.

\printbibliography

\end{document}